\documentclass[12pt]{amsart}

\usepackage{amssymb}
\usepackage{amsthm}
\usepackage{hyperref}
\usepackage[all,cmtip]{xy}

\theoremstyle{plain}
\newtheorem{theorem}{Theorem}[section]
\newtheorem{lemma}[theorem]{Lemma}
\newtheorem{corollary}[theorem]{Corollary}

\theoremstyle{definition}
\newtheorem{definition}[theorem]{Definition}

\theoremstyle{remark}
\newtheorem{remark}[theorem]{Remark}

\numberwithin{equation}{section}

\begin{document}


\title[O-minimal de Rham cohomology]{O-minimal de Rham cohomology}

\author[R. Bianconi]{Ricardo Bianconi}
\address{Instituto de Matem\'atica e Estat\'istica da Universidade de S\~ao Paulo\\
Rua do Mat\~ao, 1010, Cidade Universit\'aria, CEP 05508-090, S\~ao Paulo, SP, Brazil}
\email{bianconi@ime.usp.br}

\author[R. Figueiredo]{Rodrigo Figueiredo}
\address{Instituto de Matem\'atica e Estat\'istica da Universidade de S\~ao Paulo\\
Rua do Mat\~ao, 1010, Cidade Universit\'aria, CEP 05508-090, S\~ao Paulo, SP, Brazil}
\email{rodrigof@ime.usp.br}
\thanks{This work was developed during the second author's PhD thesis and was supported by CNPq Brazil, Proc. No.  141829/2014-1.}



\begin{abstract}
O-minimal geometry generalizes both semialgebraic and subanalytic geometries, and has been very successful in solving special cases of some problems in arithmetic geometry, such as Andr\'e-Oort conjecture. Among the many tools developed in an o-minimal setting are cohomology theories for abstract-definable continuous manifolds such as singular cohomology, sheaf cohomology and \v Cech cohomology, which have been used for instance to prove Pillay's conjecture concerning definably compact groups. In the present paper we elaborate an o-minimal de Rham cohomology theory for abstract-definable $\mathcal{C}^\infty$ manifolds in an o-minimal expansion of the real field which admits smooth cell decomposition and defines the exponential function. We can specify the o-minimal cohomology groups and attain some properties such as the existence of Mayer-Vietoris sequence and the invariance under abstract-definable $\mathcal{C}^\infty$ diffeomorphisms. However, in order to obtain the invariance of our o-minimal cohomology under abstract-definable homotopy we must, working in a tame context that defines sufficiently many primitives, assume the validity of a statement related to Br\"ocker's question.
\end{abstract}

\maketitle


\tableofcontents

\section{Introduction}

O-minimal structures have their roots in the early 80's in the work \cite{vdDries1984}. In that paper, L. van den Dries, before discussing the question raised by Tarski in his monograph \cite{Tarski} of whether the elementary theory of the exponential field $(\mathbb{R}, +,\cdot, \exp)$ is decidable, derives some finiteness properties of sets definable in an expansion of $(\mathbb{R},<)$ of finite type (\textit{i.e.}, one within the definable subsets of $\mathbb{R}$ are unions of a finite set and finitely many intervals). Soon afterwards, in a series of three papers \cite{PSI}, \cite{KPS},  and \cite{PSIII}, A. Pillay, C. Steinhorn and J. Knight give a systematized treatment of expansions of a dense linear order without endpoints that have the strong condition of ``every definable set with parameters is a finite union of intervals and points'', under the coinage of o-minimal structures, extending the work of L. van den Dries, among other things. We refer the reader to \cite{vandendries-tame} and \cite{vandendries-miller-geometric} for an introduction to o-minimal structures from a geometric viewpoint. 

O-minimality found deep connections with diophantine geometry in the first decade of 21st century, say beginning with the study carried out by Pila and Wilkie of rational points in a definable set \cite{pilawilkie2006} and culminating in the unconditional proof of Andr\'e-Oort conjecture for arbitrary products of modular curves \cite{pila2011} by Pila. Postliminary works, for instance \cite{pilatsimerman2013} and \cite{gao2016}, presenting solutions for special cases of this conjecture have also used o-minimality in a crucial way; also, in a recent paper \cite{Wilkie2016} Wilkie raises diophantine questions in the spirit of those addressed in \cite{pilawilkie2006}, which somehow shows that applications of this fragment of model theory to algebraic geometry are far from having been exhausted. 

Linked up with algebraic geometry although in a different direction, Edmundo developed a cohomology theory for the category of definable manifolds and continuous maps within the framework of an o-minimal expansion of a real closed field \cite{edmundo2001}, and used this to solve a problem, proposed by Peterzil and Steinhorn \cite{peterzil-etc1999}, concerning the existence of torsion points on definably compact definable abelian groups. (Here ``definable manifold'' is, in our parlance, an abstract-definable $\mathcal{C}^0$ manifold - see section 2.) Subsequently, Edmundo, Jones and Peatfield established a sheaf cohomology for the category of definable sets in an o-minimal expansion of a group \cite{edmundo-etc2006}; and, working in the category where the sets and continuous maps are all definable in an arbitrary o-minimal structure with definable Skolem functions, Edmundo and Peatfield proved the existence of a \v Cech cohomology theory \cite{edmundo-etc2008}. In view of all these results settling cohomologies for (abstract-)definable objects, we inquire about the existence of a definable analogue of the de Rham cohomology on a tame category. 

In the present paper we elaborate a de Rham-like cohomology theory for abstract-definable $\mathcal{C}^\infty$ manifolds in the setting of an o-minimal expansion of the real field which admits smooth cell decomposition and defines the exponential function, and show that such a cohomology has certain strong properties only in particular o-minimal contexts. Our program is to follow the lines of the construction of the classical de Rham cohomology starting from the general context of abstract-definable $\mathcal{C}^p$ manifolds ($p<\infty$) where the fixed framework is an arbitrary o-minimal expansion of a real closed field, and push it to the limit. Abstract-definable manifold of class $\mathcal{C}^p$ ($0<p<\infty$), in an o-minimal structure expanding a real closed field, generalizes the notion of an abstract $\mathcal{C}^p$ Nash manifold \cite{shiota1986}, since the transition maps might possess additional parts other than the semialgebraic.

This paper is organized as follows. 

We fix an o-minimal expansion of a real closed field and introduce in Section 2 the notion of an abstract-definable manifold $M$ of class $\mathcal{C}^p$, with $0\leq p<\infty$, and prove some basic topological facts concerning the manifold topology, some of them quite similar to the classical case, for example, every abstract-definable $\mathcal{C}^p$ manifold is definably regular, locally definably compact; and some others different such as every abstract-definable $\mathcal{C}^p$ manifold has finitely many definably connected components. Also, in this section, we establish the tangent space $T_x M$ of an abstract-definable $\mathcal{C}^p$ manifold $M$ at a point $x$ in $M$ by following \cite{ibjorgen2001}, and its corresponding cotangent space $T_x^*M$. 
Section 3 is where the main difficulty of this work lies in, and it is devoted to the construction of an abstract-definable version of partitions of unity (with respect to an abstract-definable $\mathcal{C}^p$ atlas) and some of their consequences, as the existence of abstract-definable $\mathcal{C}^p$ bump functions. Unlike in the classical setting in which a partition of unity subordinate to a fixed open cover of a smooth manifold is built upon the employment of tools such as smooth bump functions and the existence of a countable basis for a smooth manifold - both of them unavailable for us -, we adapt a method by Fischer \cite{fischer2008}, used to attain partitions of unity within the framework of an o-minimal expansion of the exponential real field that admits smooth cell decomposition, that makes heavy use of the finiteness of the atlas, and a weaker form of the definable $\mathcal{C}^p$ Urysohn's lemma having a definable open set $U\subseteq R^m$ as the background topological space (Lemma \ref{3pu}). In order to obtain this weak Urysohn's lemma, Fisher settles a result concerning the approximation of definable continuous functions by definable $\mathcal{C}^p$ functions, where $0< p\leq \infty$. Thamrongthanyalak proves \cite{Athipat2013} an analogous approximation theorem for o-minimal expansions of a real closed field, and for $0<p<\infty$, enabling the technique in question to be applied within our context. (Observe that if the concerned definable open set $U$ were the whole $R^m$, we would be done by Corollary C.12 \cite{vandendries-miller-geometric}, with no need of any approximation result at all.) 
After providing the foundations, we proceed to introduce the notion of an abstract-definable $\mathcal{C}^p$ vector bundle in Section 4, which is entirely analogous to the classical case, except that the number of the local trivializations is finite and the maps involved are abstract-definable $\mathcal{C}^p$. In a similar fashion, we bring in the concept of abstract-definable $\mathcal{C}^p$ sections, and give a local description of some of them. 
In Section 5 we present the core element in the study of the o-minimal de Rham cohomology theory, the abstract-definable $\mathcal{C}^p$ $k$-forms with $k\geq 0$, and give a characterization of these special abstract-definable $\mathcal{C}^p$ sections in terms of the coordinate frames. This is used to prove, among other things, that the pullback of abstract-definable $\mathcal{C}^p$ $k$-forms under an abstract-definable $\mathcal{C}^p$ map are  abstract-definable $\mathcal{C}^{p-1}$ $k$-forms. 
In Section 6 we turn our attention to smooth abstract-definable forms, which requires that we work in an o-minimal expansion $\mathcal{R}$ of the real field which admits cell decomposition, and also defines the exponential, provided that we want to exploit what we produced so far. Following the classical case, we verify the existence and uniqueness of an exterior derivative on the vector space of all abstract-definable $\mathcal{C}^\infty$ forms. This exterior derivative is tame in the sense that the abstract-definability of the forms is preserved. Moreover it commutes with the pullback of an abstract-definable $\mathcal{C}^\infty$ map. 
Finally, in Section 7 we specify the $k$th o-minimal de Rham cohomology groups, and demonstrate that the o-minimal de Rham cohomology satisfies all analogous main theorems for classical de Rham cohomology but the Hotomopy Axiom (Theorem \ref{8dr}), once such a statement fails for instance in the o-minimal context of the exponential real field $\mathbb{R}_{\text{exp}}$. We finish this paper by showing that the Homotopy Axiom holds, therefore so does the Poincar\'e Lemma (Corollary \ref{9dr}), when the setting is the Pfaffian closure of $\mathcal{R}$ and the Br\"ocker's question holds true for every pair $(\mathfrak{R}, \widetilde{\mathfrak{R}})$ of o-minimal expansions of $\mathbb{R}_{\exp}$, with $\widetilde{\mathfrak{R}}$ taken to be the Pfaffian closure of $\mathfrak{R}$.

We follow closely \cite{Tu2010} in the development of Sections 4-7, with no originality claimed other than the adjustments we had to make.
\\

\noindent\textsc{Acknowledgments}. The authors are grateful to Hugo Luiz Mariano and Tobias Kaiser for their valuable suggestions and insightful comments.
\\

\textit{For Sections 2-5, we fix an o-minimal expansion $\mathcal{R}$ of an arbitrary real closed field $(R,>,0,1,+,-,\cdot)$. By ``definable'' we mean ``definable in $\mathcal{R}$ with parameters in $R$'', unless otherwise stated}.
\\

\textsc{Notation}. $\mathbb{N}$ denotes the set of nonnegative integers, and $\mathbb{R}$ the field of real numbers. For any set $X$, $\text{id}_X$ denotes the identity map $x\mapsto x$ on $X$. The $m$-tuple $(r^1,\ldots,r^m)$ indicates the coordinates of a point in $R^m$ in the standard basis. Given a topological space $X$ and a subset $Y$ of $X$, by $\text{int}_X(Y),\text{cl}_X(Y)$, and $\text{bd}_X(Y)$ we mean the topological interior, closure and boundary of $Y$ in $X$ respectively; when it is clear from the context the topological space, we drop the letter $X$ in these notations. For any map $f\colon A\to B$, $\Gamma(f)$ denotes its graph $\{(x,f(x))\,:\, x\in A\}$. Given a map $f$ from an interval $I$ to a topological space $X$ and a limit point $a$ of $I$, we denote by $\lim_{t\to a^+}f(t)$ and $\lim_{t\to a^-} f(t)$ the right- and left-handed limits of $f$ at $a$, respectively. The collection of all functions from a set $X$ to $R$ will be denoted by $\mathcal{F}(X)$. ($\mathcal{F}(X)$ is made into a commutative ring with identity when endowed it with pointwise sum and multiplication.)
Further notations are explained along the text.

\section{Tame calculus on abstract-definable manifolds}

Let $M$ be a set, and let $\{\phi_i\colon U_i\to \phi_i(U_i)\subseteq R^m\}_{i\in \Lambda}$ be a finite family of set-theoretic bijections, where each $U_i$ is a subset of $M$ and $\phi_i(U_i)$ is a definable open set in $R^m$. Recall from Section 10 (\cite{beot2001}, p. 114) that such a collection is said to be an \textit{abstract-definable $\mathcal{C}^p$ atlas} on $M$ of dimension $m$, where $0\leq p<\infty$, if $M=\bigcup_{i\in \Lambda} U_i$ and for any $i,j\in \Lambda$ the sets  $\phi_i(U_i\cap U_j), \phi_j(U_i\cap U_j)$ are definable and open in $R^m$ and the map $\phi_j\circ \phi_i^{-1}\colon \phi_i(U_i\cap U_j) \to \phi_j(U_i\cap U_j)$ is a definable $\mathcal{C}^p$-diffeomorphism. (By ``definable $\mathcal{C}^0$-diffeomorphism'' we mean ``definable homeomorphism''.) The elements $\phi\colon U\to \phi(U)$ of an abstract-definable $\mathcal{C}^p$ atlas are called \textit{charts}, and will usually be written as the pair $(U,\phi)$.

The relation $\sim$, defined on the set of all abstract-definable $\mathcal{C}^p$ atlases of dimension $m$ on a set $M$ by $\mathcal{A}\sim \mathcal{B}$ if and only if $\mathcal{A}\cup \mathcal{B}$ is an abstract-definable $\mathcal{C}^p$ atlas on $M$, is an equivalence relation. In this case, we say that $\mathcal{A}$ and $\mathcal{B}$ are \textit{compatible}.
\\

\textsc{Notation}. Throughout the text, the symbol $\sim$ will designate this relation of atlas compatibility.
\\

Any abstract-definable $\mathcal{C}^p$ atlas $\{\phi_i\colon U_i\to \phi_i(U_i)\subseteq R^m\}_{i\in \Lambda}$ on a set $M$ endows such a set with a topology whose open sets are those subsets $U\subseteq M$ such that $\phi_i(U_i\cap U)$ are open in $R^m$ for all $i\in \Lambda$. This is the unique topology on $M$ in which each $U_i$ is open and every $\phi_i$ is a homeomorphism. Two $\sim$-equivalent abstract-definable $\mathcal{C}^p$ atlases on a set induce the same topology, the \textit{manifold topology}. The manifold topology is obviously $T_1$, although is not Hausdorff as it shows Example 2.5 (\cite{edmundo-etc2016}, p. 4). Namely, consider the set $M$ given by the line segment with a point $\{(x,y)\in ]c,d[^2\,:\, x=y\}\cup \{(b,a)\}$, where $c<a<b<d$ in $R$, $U_1\mathrel{\mathop:}=M\setminus \{(b,b)\}$, $U_2\mathrel{\mathop:}=M\setminus \{(b,a)\}$, and let $\phi_i$ be the bijection $\pi|_{U_i}\colon U_i\to ]c,d[$ ($i=1,2$), where $\pi\colon R^2\to R$ denotes the projection onto the first coordinate, and note that the manifold topology on $M$ does not separate the points $(b,a)$ and $(b,b)$.

An \textit{abstract-definable $\mathcal{C}^p$ manifold} of dimension $m$ is a set $M$ together with a $\sim$-equivalence class of $m$-dimensional abstract-definable $\mathcal{C}^p$ atlases on $M$, whose manifold topology is Hausdorff. By abuse of notation, we will write just a pair $(M, \mathcal{A})$ to indicate an abstract-definable $\mathcal{C}^p$ manifold, or simply the set $M$ when the abstract-definable $\mathcal{C}^p$ atlas $\mathcal{A}$ is clear from the context. 

We bring back to the reader's mind from Chapter 7 (\cite{vandendries-tame}, p. 116) that a definable map $f\colon A\to R^n$, where $A\subseteq R^m$ is not necessarily open, is called a \textit{$\mathcal{C}^p$-map} if $f$ can be extended to a definable map of class $\mathcal{C}^p$ defined on an open set.

Let $(M,\mathcal{A})$ and $(N,\mathcal{B})$ be two abstract-definable $\mathcal{C}^p$ manifolds. Recall from Section 10 (\cite{beot2001}, p. 115) that a subset $A\subseteq M$ is called an \textit{abstract-definable set} in $M$ if $\phi(U\cap A)$ is definable for every chart $(U,\phi)$ in $\mathcal{A}$; and a map $f\colon M\to N$ is said to be \textit{abstract-definable} (resp., \textit{abstract-definable $\mathcal{C}^p$}, an \textit{abstract-definable $\mathcal{C}^p$ diffeomorphism}) if for every point $x\in M$ and any charts $(U,\phi)\in \mathcal{A}$, $(V,\psi)\in \mathcal{B}$ with $x\in U$ and $f(x)\in V$ the restriction 
$$
\psi\circ f\circ \phi^{-1}|_{\phi(U\cap f^{-1}(V))}\colon \phi(U\cap f^{-1}(V))\to \psi(f(U)\cap V)
$$ 
is definable (resp., a $\mathcal{C}^p$-map, a definable $\mathcal{C}^p$ diffeomorphism). The set of all abstract-definable open sets in $M$ forms a basis for the manifold topology. Moreover, abstract-definability of sets is stable under $\sim$-equivalent abstract-definable $\mathcal{C}^p$ atlases.  

If $f\colon M\to N$ is an abstract-definable $\mathcal{C}^p$ map between abstract-definable $\mathcal{C}^p$ manifolds, then: (i) The set of all abstract-definable subsets of $M$ forms a boolean algebra; (ii) for any abstract-definable subset $A$ of $M$, its topological closure $\text{cl}(A)$, interior $\text{int}(A)$ and boundary $\text{bd}(A)$ in $M$ are also abstract-definable; (iii) for any abstract-definable subset  $A$ of $M$, $f(A)$ is abstract-definable in $N$; (iv) for any abstract-definable subset $B$ of $N$, $f^{-1}(B)$ is abstract-definable in $M$; (v) the graph $\Gamma(f)$ of $f$ is an abstract-definable subset of $M\times N$; and (vi) in the case $M$, $N$ are definable as well as the charts in $M$ and $N$, every abstract-definable subset of $M$ and all abstract-definable functions from $M$ to $N$ are definable.

Let $g=(g_1,g_2)\colon M\to M_1\times M_2$ be a map, where $M_1$ and $M_2$ are abstract-definable $\mathcal{C}^p$ manifolds. Then, $g$ is abstract-definable $\mathcal{C}^p$ if and only if $g_1\colon M\to M_1$ and $g_2\colon M\to M_2$ are also abstract-definable $\mathcal{C}^p$ maps.

Let $M$ and $N$ be abstract-definable $\mathcal{C}^p$-manifolds. If $\mathcal{A}_1\sim \mathcal{A}_2$ are abstract-definable $\mathcal{C}^p$ atlases on $M$ and $\mathcal{B}_1\sim \mathcal{B}_2$ on $N$ then every subset $A\subseteq M$ abstract-definable in $\mathcal{A}_1$ is abstract-definable with respect to $\mathcal{A}_2$, and every map $f\colon M\to N$ abstract-definable $\mathcal{C}^p$ relative to $\mathcal{A}_1$ and $\mathcal{B}_1$ is also abstract-definable $\mathcal{C}^p$ in $\mathcal{A}_2$ and $\mathcal{B}_2$. 

\begin{remark}\label{remarksection1}
Given an abstract-definable $\mathcal{C}^p$ manifold $(M,\mathcal{A})$ of dimension $m$, we may always assume that the range $\phi(U)$ of the charts $(U,\phi)$ in $\mathcal{A}$ are bounded open sets in $R^m$, because the map $\tau\colon R^m\to ]-1,1[^m$ defined as 
$$
\tau(r^1,\ldots,r^m)\mathrel{\mathop :}=\left(\frac{r^1}{\sqrt{1+(r^1)^2}},\ldots,\frac{r^m}{\sqrt{1+(r^m)^2}}\right)
$$ 
is a semi-algebraic $\mathcal{C}^p$ diffeomorphism between $R^m$ and its image, and the sets $\{(U,\tau\circ \phi)\,:\, (U,\phi)\in \mathcal{A}\}$ and $\mathcal{A}$ are $\sim$-equivalent abstract-definable $\mathcal{C}^p$ atlases on $M$. By virtue of this and Theorem 1 (\cite{Wilkie2005}, p. 4), the image of each chart in $\mathcal{A}$ is a finite union of open cells in $R^m$. Since an open cell is definably $\mathcal{C}^p$ diffeomorphic to an open box in $R^m$, which in turn is definably $\mathcal{C}^p$ diffeomorphic to $R^m$, we may also suppose, at our convenience, the image of any chart in $\mathcal{A}$ equals $R^m$. (When $R=\mathbb{R}$, the above map $\tau$ is a semialgebraic real analytic diffeomorphism onto its image.)
\end{remark}

\textsc{Notation}. From now until the end of Section 5, unless otherwise stated, $(M,\mathcal{A})$ and $(N,\mathcal{B})$ denote abstract-definable $\mathcal{C}^p$ manifolds of dimensions $m$ and $n$, respectively, with $\mathcal{A}\mathrel{\mathop :}= \{\phi_i\colon U_i\to \phi_i(U_i)\}_{i\in \Lambda}$.

\begin{definition}
We say that $M$ is \textit{definably regular} if, for any abstract-definable closed subset $F$ of $M$ and any point $x\in M\setminus F$, there are disjoint abstract-definable open subsets $U$ and $V$ of $M$ such that $x\in U$ and $V\subseteq F$.
\end{definition}

One easily sees that $M$ is definably regular if and only if for any $x\in M$ and any abstract-definable open subset $U$ of $M$ there is an abstract-definable open subset $V\subseteq M$ with $x\in V\subseteq \text{cl}(V)\subseteq U$.

The following notion of definable compactness was introduced in \cite{peterzil-etc1999}. In a Euclidean space $R^m$ this conception has a similar characterization to that of the non-first order property of compactness (\cite{peterzil-etc1999}, Theorem 2.1, p. 772), the conjunction of boundedness and closedness.  

\begin{definition}
We say that $M$ is \textit{definably compact} if for every $a,b\in R\cup \{-\infty, +\infty\}$ where $a<b$, and for every abstract-definable continuous map $\gamma\colon ]a,b[\to M$, both limits $\lim_{t\to a^+} \gamma(t)$ and $\lim_{t\to b^-} \gamma(t)$, with respect to the manifold topology, exist in $M$. We call an abstract-definable subset $K\subseteq M$ a \textit{definably compact set} if for every abstract-definable continuous map $\gamma\colon ]a,b[\to M$, with $\text{Im}\gamma\subseteq K$, the limits $\lim_{t\to a^+} \gamma(t)$ and $\lim_{t\to b^-} \gamma(t)$ exist in $K$ with respect to the subspace topology on $K$. We say that $M$ is \textit{locally definably compact} if every $x\in M$ has a definably compact neighborhood.
\end{definition}

The following appears in Corollary 2.8 (\cite{edmundo-etc2016}, p. 7) where the topological space is a generalization of an abstract-definable $\mathcal{C}^0$ manifold, namely a Hausdorff definable space (see \cite{vandendries-tame}, Definition 10.1.2, p. 156 or \cite{edmundo-etc2016}, Definition 2.1, p. 3), and the background structure is an arbitrary o-minimal structure that has definable Skolem functions. That corollary is obtained by first proving that a Hausdorff, locally definably compact definable space is definably regular. Here we give a direct proof.

\begin{lemma}\label{1}
Every definably compact set $K\subseteq M$ is closed.
\end{lemma}
\begin{proof}
We will show that $M\setminus K$ is open. Suppose, towards a contradiction, there is a point $x\in M\setminus K$ of which no open neighborhood is included in $M\setminus K$. Particularly, fixing a chart $(U,\phi)$ on $M$ at $x$, the intersection $B(\phi(x),\epsilon)\cap \phi(U)$ is not contained in $\phi(U\cap (M\setminus K))$ for each $\epsilon>0$. By definable choice, there is a definable map $\alpha\colon ]0,r[\to R^m$ such that $\alpha(t)\in  (B(\phi(x),t)\cap \phi(U))\setminus \phi(U\cap (M\setminus K))$ for all $t\in ]0,r[$. Let $\gamma\colon ]0,r[\to M$ be the composite map $\phi^{-1}\circ \alpha$. The map $\gamma$ is abstract-definable, and shrinking $r$ if necessary we may consider $\gamma$ continuous. Moreover, $\text{Im}\gamma\subseteq U\cap K$ and $\lim_{t\to 0} \gamma(t)=x$. From the definable compactness of $K$ and the uniqueness of the limit (recall that $M$ is Hausdorff), it follows that $x=\lim_{t\to 0}\gamma(t)\in K$, leading to a contradiction. 
\end{proof}

The second part of the theorem below is contained in Proposition 2.7 (\cite{edmundo-etc2016}, p. 6). Despite we also achieve the definable regularity of the abstract-definable ($\mathcal{C}^p$) manifold through the local definable compactness, our proof is rather distinct.

\begin{theorem}\label{2}
$M$ is locally definably compact and definably regular.
\end{theorem}
\begin{proof}
Fix a point $x\in M$ and an abstract-definable open $W$ in $M$ containing $x$. Pick a chart $(U,\phi)$ on $M$ at $x$. Hence, there is an open box $B\subseteq R^m$ with  $\phi(x)\in B\subseteq \text{cl}(B)\subseteq \phi(U\cap W)$. Set $K= \phi^{-1}(\text{cl}(B))\subseteq U\cap W$, an abstract-definable set whose interior contains $x$, and let $\gamma\colon ]a,b[\to M$ be an abstract-definable continuous map with $\gamma(]a,b[)\subseteq K$. The map $\phi\circ \gamma\colon ]a,b[\to \phi(U)$ is then a definable continuous curve such that $(\phi\circ \gamma)(]a,b[)\subseteq \text{cl}(B)$, and as consequence of the definable compactness of $\text{cl}(B)$ both limits $\lim_{t\to a^+} \phi(\gamma(t))$, $\lim_{t\to b^-}\phi(\gamma(t))$ exist in $\text{cl}(B)$. By setting $L_1\mathrel{\mathop:}= \phi^{-1}(\lim_{t\to a^+}\phi(\gamma(t)), L_2\mathrel{\mathop:}= \phi^{-1}(\lim_{t\to b^-}\phi(\gamma(t)))\in K$, and noticing that $\lim_{t\to a^+}\gamma(t)=L_1$ and $\lim_{t\to b^-}\gamma(t)=L_2$, we conclude that $K$ is a definably compact neighborhood of $x$ contained in $W$. This proves the first part of the theorem. The second follows from the fact that $K$ is closed in $M$ (Lemma \ref{1}), and hence $x\in \phi^{-1}(B)\subseteq \text{cl}(\phi^{-1}(B))\subseteq K\subseteq W$.
\end{proof}

\begin{definition}
We say that $M$ is \textit{definably normal} if, for any two disjoint abstract-definable closed subsets $F_1$ and $F_2$ of $M$, there are disjoint abstract-definable open subsets $U_1$ and $U_2$ such that $F_1\subseteq U_1$ and $F_2\subseteq U_2$.
\end{definition}

Equivalently, $M$ is definably normal if given two disjoint abstract-definable closed subsets $F_1,F_2\subseteq M$ there exists an abstract-definable open subset $W\subseteq M$ satisfying $F_1\subseteq W\subseteq \text{cl}(W)\subseteq M\setminus F_2$.

As pointed out in Remark 3.4 (\cite{edmundo2001}, p. 9), the abstract-definable $\mathcal{C}^p$  manifold $M$ is definably regular (see Theorem \ref{2}) and therefore, by Theorem 10.1.8 (\cite{vandendries-tame}, p. 159), there is a continuous injective map $h$ from $M$ into $R^{2+2m}$, where $m$ is the dimension of $M$, which maps $M$ homeomorphically onto the definable set $h(M)$. The definable normality of $h(M)$ can thus be transferred to $M$ via $h$. This proves the following. 

\begin{theorem}\label{3}
$M$ is definably normal.
\end{theorem}

\begin{definition}
An abstract-definable subset $S$ of $M$ is called \textit{definably connected in $M$} if there are no abstract-definable open disjoint subsets $U$ and $V$ of $M$ in such a way that $U\cap S$ and $V\cap S$ are nonempty and $S\subseteq U\cup V$. We say that $M$ is \textit{definably connected} if its underlying set is definably connected in $M$. A \textit{definably connected component} of a nonempty abstract-definable set $S\subseteq M$ is a maximal definably connected subset of $S$ in $M$. 
\end{definition}

Theorem \ref{4} below is an abstract-definable version of Proposition 3.2.18 (\cite{vandendries-tame}, p. 57).

\begin{theorem}\label{4}
The abstract-definable $\mathcal{C}^p$ manifold $M$ has finitely many definably connected components. They form a finite partition of $M$, and consequently are open and closed in $M$.
\end{theorem}
\begin{proof}
 
 Let $\mathcal{A}=\{(U_i,\phi_i)\,:\, i=1,\ldots,k\}$. Since the subsets $\phi_1(U_1)$, ..., $\phi_i(U_i\setminus \bigcup_{j<i} U_j)$, ..., $\phi_k(U_k\setminus \bigcup_{j=1}^{k-1}U_j)$ of $R^m$ are definable, there is a cell decomposition $\mathfrak{C}$ of $R^m$ partitioning them. We claim that for each $i=2,\ldots,k$ and for any cells $C\subseteq  \phi_1(U_1)$ and $C'\subseteq \phi_i(U_i\setminus \bigcup_{j<i}U_j)$ in $\mathfrak{C}$, the sets $$
 D\mathrel{\mathop :}=\phi_1^{-1}(C)\subseteq U_1\ \text{and}\ D'\mathrel{\mathop :}= \phi_i^{-1}(C')\subseteq U_i\setminus \bigcup_{j<i}U_j
 $$ 
 are definably connected in $M$. First, note that $D$ and $D'$ are abstract-definable, inasmuch as for any chart $(U_l,\phi_l)\in \mathcal{A}$ the sets 
 $$
 \phi_l(U_l\cap D)=(\phi_l\circ \phi_1^{-1})(\phi_1(U_l\cap U_1)\cap C)
 $$ 
 and 
 \[
    \phi_l(U_l\cap D')= 
\begin{cases}
    (\phi_l\circ \phi_i^{-1})(\phi_i(U_l\cap U_i)\cap C'),& \text{if } i\leq l\\
    \emptyset,& \text{if } i>l
\end{cases}
\]
 are all definable. Furthermore, if $A$ and $B$ are  abstract-definable disjoint open subsets of $M$ with $D\subseteq A\cup B$, then since $\phi_1(U_1\cap A)$ and $\phi_1(U_1\cap B)$ are disjoint open definable subsets of $\phi_1(U_1)$ covering $C$, we have without loss of generality that $C\subseteq \phi_1(U_1\cap A)$, and consequently $D\subseteq U_1\cap A\subseteq A$. A similar argument holds for $D'$. Therefore, we obtain a partition $\mathfrak{D}\mathrel{\mathop:}= \{D_1,\ldots,D_s\}$ of $M$ into definably connected sets in $M$, where each element of $\mathfrak{D}$ is either of the form $\phi_1^{-1}(C)$ for some cell $C\in \mathfrak{C}$ included in $\phi_1(U_1)$, or of the form $\phi_i^{-1}(C')$ for some $i\in \{2,\ldots,k\}$ and a cell $C'\in \mathfrak{C}$ included in $\phi_i(U_i\setminus \bigcup_{j<i} U_j)$. For each set of indices $\Lambda\subseteq \{1,\ldots,s\}$, define $D_\Lambda\mathrel{\mathop:}= \bigcup_{i\in \Lambda} D_i$, and let $D^*$ be a maximal abstract-definable set with respect to the definable connectedness, among the $2^s-1$ nonempty sets $D_\Lambda$. Note that to conclude $D^*$ is a definably connected component of $M$, the subsequent claim suffices.
 \\

\noindent\textbf{Claim 1}.\label{6,81}
If $Y\subseteq M$ is a definably connected set in $M$ with $Y\cap D^*\neq \emptyset$, then $Y\subseteq D^*$.

Consider the abstract-definable set
$$
D_Y\mathrel{\mathop:}= \underset{D_i\cap Y\neq \emptyset}{\bigcup\limits_{D_i\in \mathfrak{D}}}D_i.
$$
Observe that $Y\subseteq D_Y$, since $\mathfrak{D}$ covers $M$. If $A$ and $B$ are disjoint abstract-definable open subsets of $M$ so that $D_Y\subseteq A\cup B$, then because $Y$ is definably connected, we may assume that $Y\subseteq A$ without loss of generality. This implies that each $D_i$ in $\mathfrak{D}$ with $D_i\cap Y\neq \emptyset$ intersects $A$, and since $D_i$ is definably connected in $M$, $D_i\subseteq A$. Hence, $D_Y\subseteq A$. In other words, $D_Y$ is definably connected. Finally, note that
$$
D^*,D_Y\supseteq D^*\cap D_Y\supseteq D^*\cap Y\neq \emptyset,
$$
\textit{i.e.}, $D^*$ and $D_Y$ have a point in common. Then, $D^*\cup D_Y$ is a set of the form $D_{\Lambda}$, for some $\Lambda\subseteq \{1,\ldots,s\}$, which is definably connected in $M$ and contains $D^*$. By the maximality of $D^*$, we get $D^*=D^*\cup D_Y$, and hence $Y\subseteq D_Y\subseteq D^*$.
\\

We now draw the reader's attention to the fact that the maximal definably connected sets $D^*$ as above form a finite partition of $M$. Clearly, there are finitely many of those sets, in total. Moreover, since $\mathfrak{D}$ covers $M$ and each of its elements is contained in such a maximal definably connected set $D^*$, these sets then cover $M$. Lastly, Claim 1 implies that the sets $D^*$ are pairwise disjoint.

For the ending part of the proposition statement, note that the closure in $M$ of a definably connected set in $M$ is definably connected as well, and hence by the maximality of the definably connected components these are closed subsets of $M$. Let $\{D^*_1,\ldots,D^*_t\}$ be a partition of $M$ into definably connected components. Since for each $j$  
$$
D^*_j=M\setminus (D^*_1\cup \cdots \cup D^*_{j-1}\cup D^*_{j+1}\cup \cdots \cup D^*_t),
$$ 
it follows that $D^*_j$ is open in $M$. 
\end{proof}

The subsequent theorem, among other things, is used to compute the $0$th de Rham cohomology group of an abstract-definable manifold (see Theorem \ref{2dr}). 

\begin{theorem}\label{6,9}
A locally constant abstract-definable map $f\colon M\to N$ is constant whenever $M$ is definably connected.
\end{theorem}
\begin{proof}
It suffices to show that $f$ is constant on each definably connected component of $M$. To see this, first note that $f$ is continuous. For any definably connected component $C$ of $M$ and a fixed point $x$ of $C$, it follows from the local constancy of $f$ that $f^{-1}(c)$ is a union of open sets in $M$, where $c$ denotes the value $f(x)$. On the other hand, since $\{c\}$ is an abstract-definable closed set in $N$, $f^{-1}(c)$ is abstract-definable closed in $M$. Because $C$ is definably connected, we thus get $C\subseteq f^{-1}(c)$.
\end{proof}

Our approach to the construction of the tangent space is the same as in Chapter 9 (\cite{ibjorgen2001}, pp. 65-68).

Fix $x\in M$ and consider the set $\mathcal{C}^p(x)$ of all abstract-definable $\mathcal{C}^p$ maps $\alpha\colon I\to M$, where $I\subseteq R$ is an open interval containing $0$ and $\alpha(0)=x$, on which we have an equivalence relation 
$$
\alpha_1\sim_c \alpha_2\stackrel{\text{def}}{\Leftrightarrow} (\phi\circ \alpha_1)'(0)=(\phi\circ \alpha_2)'(0),
$$
for some $(U,\phi)$ chart on $M$ at $x$. By virtue of the chain rule for definable maps, we may replace the condition ``for some chart on $M$ at $x$'' with ``for any chart on $M$ at $x$'' in the definition  of $\sim_c$. The quotient set $\mathcal{C}^p(x)/\sim_c$ is denoted by $T_x M$.  

If $(U,\phi)$ is a chart on $M$ at $x$, the induced map $\Phi_x\colon T_x M\to R^m$ defined as $[\alpha]\mapsto (\phi\circ \alpha)'(0)$ is bijective, and hence there is a unique $R$-vector space structure on $T_x M$ which makes $\Phi_x$ into a linear isomorphism, namely: $v+w\mathrel{\mathop:}=\Phi_x^{-1}(\Phi_x(v)+\Phi_x(w))$ and $rv\mathrel{\mathop:}=\Phi_x^{-1}(r\Phi_x(v))$, for $v,w\in T_x M, r\in R$. These operations are independent of the choice of $(U,\phi)$. The set $T_x M$ together with such a linear structure is called the \textit{tangent space} to $M$ at $x$ and its elements are said to be the \textit{tangent} \textit{vectors} to $M$ at $x$.

An abstract-definable $\mathcal{C}^p$ map $f\colon M\to N$ induces at each point $x\in M$ a linear map $d_x f\colon T_x M\to T_{f(x)}N$, the \textit{differential} of $f$ at $x$, by setting $d_x f ([\alpha])\mathrel{\mathop:}=[f\circ \alpha]$. Under the identification $T_x R^m\equiv R^m$, we obtain $d_x \phi=\Phi_x$. 

An immediate consequence of the definition of the differential of an abstract-definable $\mathcal{C}^p$ map at a point is the chain rule for abstract-definable $\mathcal{C}^p$ maps.

\begin{theorem}
Let $P$ be an abstract-definable $\mathcal{C}^p$ manifold, and let $f\colon M\to N$ and $g\colon N\to P$ be abstract-definable $\mathcal{C}^p$ maps. Then $g\circ f\colon M\to P$ is abstract-definable $\mathcal{C}^p$, and for any point $x$ in $M$ we have 
$$
d_x(g\circ f) = d_{f(x)} g \circ d_x f.
$$
\end{theorem}

Given a chart $(U,\phi)$ at a point $x\in M$, the set $\{\partial/\partial x^1|_x,\ldots,\partial/\partial x^m|_x\}$ forms a basis for $T_x M$, where $\partial /\partial x^i|_x$ is $(d_x \phi)^{-1}(e_i)$ and $e_i$ denotes the $i$th standard basis vector $(0,\ldots,1,\ldots,0)$ of $R^m$. Hence, a tangent vector $X_x\in T_x M$ can be uniquely written as $X_x=\sum_{i=1}^ma_i(\partial/\partial x^i|_x)$, with $(a_1,\ldots,a_m)\in R^m$. If $X_x=[\alpha]$, for some $\alpha\in \mathcal{C}^p(x)$, then $(a_1,\ldots,a_m)=(\phi\circ \alpha)'(0)$. 

Let $f\colon M\to R$ be an abstract-definable $\mathcal{C}^p$ function. The \textit{directional derivative} $X_x f$ of $f$ at $x\in M$ is defined to be $(f\circ \alpha)'(0)$. If $(U,\phi)$ is a chart at $x$ then applying the chain rule (for definable maps) to $(f\circ \phi^{-1}\circ \phi\circ \alpha)'(0)$, we get  
$$
X_x f=\sum_{i=1}^ma_i(\partial(f\circ \phi^{-1})/\partial r^i) (\phi(x)),
$$ 
where $a_i$ are the components of $X_x$ in the basis $\{\partial/\partial x^i|_x\}_i$. Particularly, 
$$
(\partial/\partial x^i|_x)f=(\partial(f\circ \phi^{-1})/\partial r^i)(\phi(x)).
$$

The disjoint union of all tangent spaces $TM\mathrel{\mathop:}=\bigcup_{x\in M}\{x\}\times T_x M$ is called the \textit{tangent bundle} of $M$. The set $TM$ can be made into an abstract-definable $\mathcal{C}^p$ manifold of dimension $2m$ as follows. Let $(U,\phi)$ be a chart on $M$ and denote by $TU$ the disjoint union $\bigsqcup_{x\in U} T_x M$. The set of maps $\widetilde{\phi}\colon TU\to \phi(U)\times R^m$, given by 
$$
(x,\sum_{i=1}^m a_i \partial/\partial x^i|_x)\mapsto (\phi(x),a_1,\ldots,a_m),
$$ 
forms an abstract-definable $\mathcal{C}^p$ atlas on $TM$. Therefore, the projection $\pi\colon TM\to M\colon (x,v)\mapsto x$ is an abstract-definable $\mathcal{C}^p$ map.

The \textit{cotangent space} of $M$ at a point $x\in M$, $T^*_xM$, is the dual vector space of the tangent space $T_x M$ and its elements are called \textit{covectors} at $x$. The disjoint union of all cotangent spaces of $M$ is said to be the \textit{cotangent bundle} of $M$ and is denoted by $T^* M$. Just like the tangent bundle, the cotangent bundle of $M$ can be endowed with an abstract-definable $\mathcal{C}^p$ atlas of dimension $2m$, described as follows. After fixing a chart $(U,\phi)$ on $M$ at $x$, we let $\{dx^1|_x,\ldots,dx^m|_x\}$ denote the dual basis of $\{\partial/\partial x^1|_x,\ldots,\partial/\partial x^m|_x\}$ for $T_x^*M$. The set of the induced bijections 
$$
(x,\sum_{i=1}^m b_idx^i|_x)\mapsto (\phi(x),b_1,\ldots,b_m)\colon T^*U\to \phi(U)\times R^m
$$ 
then forms an abstract-definable $\mathcal{C}^p$ atlas on $T^*M$, where $T^*U$ denotes the disjoint union $\bigsqcup_{x\in U} T^*_x M$. Likewise, the natural projection $\pi\colon T^*M\to M$ turns out to be an abstract-definable $\mathcal{C}^p$ map.

\section{Abstract-definable partition of unity}

This section is devoted to the construction of an abstract-definable $\mathcal{C}^p$ partition of unity subordinate to a given abstract-definable $\mathcal{C}^p$ atlas, and some of its consequences whose classical analogues are widely known. The strategy adopted here is that of Fischer \cite{fischer2008}. 

Using Generalized Lojasiewicz Inequality (\cite{vandendries-miller-geometric}, Theorem C.14) and a stratification of definable sets where the functions involved in the strata have bounded gradient, Thamrongthanyalak obtains a result on smoothing of definable continuous functions, stated as follows.

\begin{theorem}[Theorem 1.1, \cite{Athipat2013}, p. 2]\label{2pu}
Let $f\colon U\to R$ be a definable continuous function, with $U$ open in $R^n$. Let $Z$ be a definable closed subset of $U$ such that $\dim\, Z<\dim\, U$, and $f|_{(U\setminus Z)}$ is $\mathcal{C}^p$, where $p\geq 1$. Let $\epsilon \colon U\to ]0,+\infty[$ be a definable continuous function. Then, for any definable neighbourhood $V$ of $Z$ in $U$, there is a definable $\mathcal{C}^p$ function $g\colon U\to R$ such that
\begin{enumerate}
\item[(1)] $|g(x)-f(x)|<\epsilon(x)$, for every $x\in U$;
\item[(2)] $g=f$ outside $V$.
\end{enumerate}
\end{theorem}

If $f\colon U\to R$ is a definable function defined on an open set $U\subseteq R^m$ then from the $\mathcal{C}^p$-cell decomposition it follows that the dimension of the closure in $U$ of the definable set comprised of the points in $U$ at which $f$ is not $\mathcal{C}^p$ is strictly less than that of $U$. 

The subsequent lemma is Corollary 1.2 in \cite{fischer2008} (p. 497) whose proof was adjusted to our case.

\begin{lemma}\label{3pu}
Let $U\subseteq R^m$ be a definable open set, and let $A,B\subseteq U$ be definable disjoint sets, which are closed in $U$. Then, there is a definable $\mathcal{C}^p$ function $f\colon U\to R$ such that $A\subseteq \{f=1\}$ and $B\subseteq \{f=0\}$.
\end{lemma}
\begin{proof}
Since $U$ is definably normal, there are definable open sets $V_A$ and $V_B$ in $U$ such that $A\subseteq V_A\subseteq \text{cl}_U(V_A)\subseteq U\setminus B$ and $B\subseteq V_B\subseteq \text{cl}_U(V_B)\subseteq U\setminus A$. In particular, $\text{cl}_U(V_A)\cap \text{cl}_U(V_B)=\emptyset$. Consider $g\colon U\to R$ a definable continuous function with $0\leq g\leq 1$, $g^{-1}(1)=\text{cl}_U(V_A)$, and $g^{-1}(0)=\text{cl}_U(V_B)$ (Lemma 6.3.8, \cite{vandendries-tame},  p. 102), and let $C$ be the definable set of points in $U$ at which $g$ is not $\mathcal{C}^p$. Because $C$ is contained in $U\setminus (V_A\cup V_B)$, we get $\text{cl}_U(C)\cap (A\cup B)=\emptyset$. Also, $\text{cl}_U(C)<\dim\, U$ (see the observation above this lemma). Thus, by Lemma \ref{2pu}, there is a definable $\mathcal{C}^p$ function $f\colon U\to R$ such that $f=g$ in $A\cup B$. 
\end{proof}

\begin{theorem}\label{4pu}
There exist abstract-definable $\mathcal{C}^p$ functions $\varphi_i\colon M\to R$ such that $\varphi_i\geq 0$, $\text{supp}(\varphi_i)\subseteq U_i$, and $\sum_{i\in \Lambda}\varphi_i=1$, for each $i\in \Lambda$. The collection $\{\varphi_i\}_{i\in \Lambda}$ is called an \emph{abstract-definable $\mathcal{C}^p$ partition of unity subordinate to $\{U_i\}_{i\in \Lambda}$}.
\end{theorem}
\begin{proof}
For the sake of simplicity, let us assume without loss of generality that $\Lambda=\{1,2\}$. By virtue of Lemma \ref{4,1pu} below, define each $\varphi_i\colon M\to R$ as 
$\varphi_i \mathrel{\mathop:}=\psi_i/(\psi_1+\psi_2)$. It is readily seen that these functions have the above required properties.
\end{proof}

\begin{lemma}\label{4,1pu}
There exist abstract-definable $\mathcal{C}^p$ nonnegative functions $\psi_1,\psi_2\colon M\to R$ satisfying $\text{supp}(\psi_i)\subseteq U_i$ and $M=\psi_1^{-1}(]0,+\infty[)\cup \psi_2^{-1}(]0,+\infty[)$.
\end{lemma}
\begin{proof}
Consider $V_1$ the abstract-definable closed subset $U_1-U_2$ of $M$ which does not intersect the abstract-definable closed subset $\text{bd}(U_1)$. Denote by $\Omega_1, \Omega_2$ the disjoint abstract-definable open sets in $M$ such that $V_1\subseteq \Omega_1$ and $\text{bd}(U_1)\subseteq \Omega_2$. (Recall that $M$ is definably normal). Let $W_1$ be the intersection $\text{cl}(\Omega_2)\cap U_1$. Since $\Omega_1\cap \text{cl}(\Omega_2)=\emptyset$, the (abstract-definable) closed subsets $V_1$ and $W_1$ of $U_1$ are disjoint. Consequently, $\phi_1(V_1)$ and $\phi(W_1)$ are disjoint definable closed subsets of $\phi_1(U_1)$. By Lemma \ref{3pu}, there is a definable $\mathcal{C}^p$ function $f_1\colon \phi_1(U_1)\to R$ such that $V_1\subseteq \{f_1=1\}$ and $W_1\subseteq \{f_1=0\}$. Squaring if necessary, we may assume that $f_1\geq 0$. Take $\psi_1\colon M\to R$ to be the nonnegative function given by 
$$
\psi_1\mathrel{\mathop:}=
\begin{cases}
f_1\circ \phi_1& \text{on } U_1\\
0& \text{on } M\setminus U_1
\end{cases}
$$
In order to obtain $\text{supp}(\psi_1)\subseteq U_1$ it suffices to prove that the set $\text{cl}(\phi^{-1}_1(\{f_1\neq 0\}))$ does not intersect $\text{bd}(U_1)$. But this follows immediately from the inclusions 
\begin{align*}
\phi_1^{-1}(\{f_1\neq 0\})&=U_1\setminus\phi_1^{-1}(\{f_1=0\})\subseteq U_1\setminus W_1\\
&=U_1\setminus \text{cl}(\Omega_2)\subseteq U_1\setminus \Omega_2\subseteq M\setminus \Omega_2
\end{align*}
and $\text{bd}(U_1)\subseteq \Omega_2$. From the inclusion $\text{supp}(\psi_1)\subseteq U_1$ we can easily conclude that $\psi_1$ is an abstract-definable $C^p$ function. Proceeding in a similar way for $V_2= U_2\setminus \psi_1^{-1}(]0,+\infty[)$ and $W_2= \text{cl}(\Theta_2)$, where $\Theta_2$ is an abstract-definable open subset of $M$ containing $\text{bd}(U_2)$ whose existence is ensured by the definable normality of $M$, we may construct an abstract-definable $\mathcal{C}^p$ nonnegative function $\psi_2\colon M\to R$ satisfying $\text{supp}(\psi_2)\subseteq U_2$. 
Finally, note that the sets $\{\psi_1>0\}$ and $\{\psi_2>0\}$ cover $M$, by the construction of the functions $\psi_i$.  
\end{proof}

\begin{corollary}\label{5pu}
Let $\{V, W\}$ be an abstract-definable open cover of $M$. There are abstract-definable $\mathcal{C}^p$ nonnegative functions $f_V,f_W\colon M\to R$ such that $\text{supp}(f_V)\subseteq V$, $\text{supp}(f_W)\subseteq W$, and $f_V+f_W=1$.
\end{corollary} 
\begin{proof}
Consider $\phi_i^V$ and $\phi_i^W$ the restrictions ${\phi_i|}_{V_i}$ and ${\phi_i|}_{W_i}$, respectively, where $V_i\mathrel{\mathop:}=U_i\cap V$ and $W_i\mathrel{\mathop:}=U_i\cap W$, for each $i\in \Lambda$. The collection $\{\phi_i^V\colon V_i\to \phi_i(V_i),\phi_i^W\colon W_i\to \phi_i(W_i)\}_{i\in \Lambda}$
is then an abstract-definable $\text{C}^p$ atlas on $M$, $\sim$-equivalent to $\mathcal{A}$. Applying Theorem \ref{4pu} to this atlas, we obtain abstract-definable $\mathcal{C}^p$ nonnegative functions $\varphi_i^V\colon M\to R$, $\varphi_i^W\colon M\to R$ ($i\in \Lambda$) satisfying the conditions $\text{supp}(\varphi_i^V)\subseteq V_i$, $\text{supp}(\varphi_i^W)\subseteq W_i$, and $\sum_{i\in \Lambda}\varphi_i^V+\sum_{i\in \Lambda}\varphi^W_i=1$. For the conclusion, it suffices to define $f_V\colon M\to R$ and $f_W\colon M\to R$ to be  
$f_V\mathrel{\mathop:}= \sum_{i\in \Lambda}\varphi_i^V$, and $f_W\mathrel{\mathop:}= \sum_{i\in \Lambda}\varphi_i^W$ respectively.
\end{proof}

The following is the abstract-definable $\mathcal{C}^p$ version of the Urysohn's lemma.

\begin{corollary}\label{6pu}
Let $A$ and $B$ be disjoint abstract-definable closed sets in $M$. Then, there exists an abstract-definable $\mathcal{C}^p$ nonnegative function $f\colon M\to R$ which is identically $1$ on $A$, and identically $0$ on $B$.
\end{corollary}
\begin{proof}
Apply Corollary \ref{5pu} to the abstract-definable open cover $\{V,W\}$ of $M$, where $V$ denotes $M\setminus A$ and $W$ denotes $M\setminus B$, to obtain abstract-definable $\mathcal{C}^p$ nonnegative functions $f_V,f_W\colon M\to R$ such that $\text{supp}(f_V)\subseteq V$, $\text{supp}(f_W)\subseteq W$, and $f_V+f_W=1$. Then by letting $f\colon M\to R$ be $f_W$, the result thus follows. 
\end{proof}

\begin{corollary}\label{7pu}
Let $F$ be an abstract-definable closed set in $M$, and $U$ an abstract-definable open set in $M$ containing $F$. Then, there exists an abstract-definable $\mathcal{C}^p$ function $\rho\colon M\to R$ so that $0\leq \rho\leq 1$, $\rho|_F=1$, and $\text{supp}(\rho)\subseteq U$.
\end{corollary}
\begin{proof}
Let $\{M\setminus F,U\}$ be an abstract-definable open cover of $M$. By Corollary \ref{5pu}, there are abstract-definable $\mathcal{C}^p$ nonnegative functions $f,g\colon M\to R$ which have the properties $\text{supp}(f)\subseteq M\setminus F$, $\text{ supp}(g)\subseteq U$, and $f+g=1$. By defining $\rho$ to be $g$, we are done. 
\end{proof}

\begin{corollary}\label{8pu}
For any abstract-definable open subset $U\subseteq M$, there is an abstract-definable $\mathcal{C}^p$ nonnegative function $\rho\colon M\to R$ such that ${\rho|}_V=1$, for some abstract-definable open set $V\subseteq U$, and $\text{supp}(\rho)\subseteq U$.
The function $\rho$ is called an \emph{abstract-definable $\mathcal{C}^p$ bump function supported in $U$}.
\end{corollary}
\begin{proof}
Fix a point $x$ in $U$. Since $M$ is definably regular, there is an abstract-definable open set $V$ with $x\in V\subseteq \text{cl}(V)\subseteq U$.
By applying Corollary \ref{7pu} to $\text{cl}(V)$, we immediately obtain the desired function $\rho$. 
\end{proof}

It is quite satisfactory in verifying that a map is abstract-definable $\mathcal{C}^p$ to choose only convenient charts. This is what the following states.

\begin{lemma}\label{2f*}
A map $f\colon M\to N$ is abstract-definable $\mathcal{C}^p$ if and only if for each $x\in M$ there is a chart $(U,\phi)$ on $M$ at $x$ and a chart $(V,\psi)$ on $N$ at $f(x)$ such that the restriction of $\psi\circ f\circ \phi^{-1}$ to $\phi(U\cap f^{-1}(V))$ is a $\mathcal{C}^p$-map.
\end{lemma}
\begin{proof}
The ``only if'' direction is immediate. For the ``if'' direction, fix a point $x$ in $M$, and let $(U,\phi)$, $(V,\psi)$ be arbitrary charts respectively on $M$ at $x$ and on $N$ at $f(x)$. We must prove that the restriction of $\psi\circ f\circ \phi^{-1}$ to $\phi(U\cap f^{-1}(V))$ is a $\mathcal{C}^p$-map. This will be done first by showing that the concerned restricted map is definable, and then it is extendable to a definable map of class $\mathcal{C}^p$ defined on an open definable set. For each $z\in U\cap f^{-1}(V)$, pick a chart $(U_z,\phi_z)\in \mathcal{A}$ with $z\in U_z$, and a chart $(V_z,\psi_z)\in \mathcal{B}$ with $f(z)\in V_z$ in such a way that $\psi_z\circ f\circ \phi_z^{-1}|_{\phi_z(U_z\cap f^{-1}(V_z))}$ is a $\mathcal{C}^p$-map. Since the set of these chosen charts is contained in $\mathcal{A}\cup \mathcal{B}$, $U\cap f^{-1}(V)$ can be expressed as a finite union
$$
U\cap f^{-1}(V)=\bigcup\limits_{\alpha\in \Lambda} (U\cap f^{-1}(V))\cap (U_\alpha\cap f^{-1}(V_\alpha)),
$$
where $\Lambda$ is an enumeration of this set of the chosen charts. Consequently, 
$$
\phi(U\cap f^{-1}(V))=\bigcup\limits_{\alpha\in \Lambda} \phi(U\cap f^{-1}(V)\cap U_\alpha\cap f^{-1}(V_\alpha)).
$$
Note that on each definable set $\phi(U\cap f^{-1}(V)\cap U_\alpha\cap f^{-1}(V_\alpha))$ the map $\psi\circ f\circ \phi^{-1}$ equals the $\mathcal{C}^p$-map 
$$
(\psi\circ \psi_\alpha^{-1})\circ (\psi_\alpha\circ f\circ \phi_\alpha^{-1})\circ (\phi_\alpha\circ \phi^{-1}).
$$ 
Therefore,  $\psi\circ f\circ \phi^{-1}|_{\phi(U\cap f^{-1}(V)\cap U_\alpha\cap f^{-1}(V_\alpha))}$ is a $\mathcal{C}^p$-map and the restriction $\psi\circ f\circ \phi^{-1}|_{\phi(U\cap f^{-1}(V))}$ is definable. Now put
$$
g_\alpha\mathrel{\mathop :}= {\psi\circ f\circ \phi^{-1}|}_{\phi(U\cap f^{-1}(V)\cap U_\alpha\cap f^{-1}(V_\alpha))}.
$$
By the definition of $\mathcal{C}^p$-map, for each $\alpha$ there is a definable $\mathcal{C}^p$ map $\widetilde{g}_\alpha\colon W_\alpha\to R^m$, with $W_\alpha$ definable open subset of $R^m$ containing the definable set $\phi(U\cap f^{-1}(V)\cap U_\alpha\cap f^{-1}(V_\alpha))$, which extends $g_\alpha$. Set $W\mathrel{\mathop :}= \bigcup_{\alpha\in \Lambda} W_\alpha$. Observe that $W$ is a definable open set in $R^m$ containing each set $\phi(U\cap f^{-1}(V)\cap U_\alpha\cap f^{-1}(V_\alpha))$. Also, $(W,\{\text{id}_{W_\alpha}\colon W_\alpha\to W_\alpha\}_{\alpha\in \Lambda})$ is an abstract-definable $\mathcal{C}^p$ manifold. Theorem \ref{4pu} thus ensures the existence of abstract-definable $\mathcal{C}^p$ functions $\varphi_\alpha\colon W\to R$, $\alpha\in \Lambda$, satisfying the conditions: $\varphi_\alpha\geq 0$, $\text{supp}(\varphi_\alpha)\subseteq W_\alpha$, and $\sum_{\alpha\in \Lambda}\varphi_\alpha=1$. Because the underlying set $W$ and the charts $\text{id}_{W_\alpha}$ are all definable, the functions $\varphi_\alpha$ are also definable. Define $\widetilde{g}\colon W\to R^n$ as 
$$
\widetilde{g}(x)\mathrel{\mathop :}= \sum\limits_{\alpha\in \Lambda} \varphi_\alpha(x)\cdot\widetilde{g}_\alpha(x),
$$
and note that in addition to being definable $\mathcal{C}^p$, $\widetilde{g}$ also agrees with $\psi\circ f\circ \phi^{-1}$ on $\phi(U\cap f^{-1}(V))$.
\end{proof}

\section{Abstract-definable vector bundles}

\begin{definition}\label{1vb}
Let $\pi\colon E\to M$ be an abstract-definable $\mathcal{C}^p$ map between abstract-definable $\mathcal{C}^p$ manifolds satisfying the conditions:
\begin{enumerate}
\item[(i)] for every $x\in M$ the \textit{fiber} at $x$, $E_x\mathrel{\mathop:}= \pi^{-1}(x)$, has the structure of a $d$-dimensional $R$-vector space;
\item[(ii)] $M$ has a finite abstract-definable open cover $\{\Omega_j\}_{j\in J}$ and for each $j\in J$ there exists an abstract-definable $\mathcal{C}^p$ diffeomorphism $\varphi_j\colon \pi^{-1}(\Omega_j)\to \Omega_j\times R^d$ such that $\text{pr}\circ \varphi_j=\pi$ on $\pi^{-1}(\Omega_j)$, where $\text{pr}$ denotes the set-theoretic projection on the first factor $(x,y)\mapsto x$,  and for each $x\in \Omega_j$ the map $\varphi_j|_{E_x}\colon E_x \to \{x\}\times R^d$ is a linear isomorphism.
\end{enumerate}
The triple $(E,M,\pi)$ is then said to be an \textit{abstract-definable $\mathcal{C}^p$ vector bundle} of rank $d$, $E$ the \textit{total space}, and $M$ the \textit{base space}. Also, the collection $\{(\Omega_j,\varphi_j)\}_{j\in J}$ is called a \textit{local trivialization} for $E$ and $\{\Omega_j\}_{j\in J}$ a \textit{trivializing open cover} of $M$.
\end{definition}

 As an abuse of notation, we will often denote an abstract-definable $\mathcal{C}^p$ vector bundle $(E,M, \pi)$ by simply $E$ or $\pi\colon E\to M$.

The tangent and cotangent bundles of $M$ together with their respective projections onto $M$ are the most well known examples of abstract-definable $\mathcal{C}^p$ vector bundles. The fibers at each point of $M$ are respectively the tangent and cotangent spaces. The triple $(M\times R^d, M, \pi)$, where $\pi\colon M\times R^d\to M$ is the projection onto $M$, is an abstract-definable $\mathcal{C}^p$ bundle of rank $d$, called the \textit{trivial vector bundle}. The fiber at every point in $M$ is just the vector space $R^d$.

\begin{definition}\label{2vb}
Let $\pi\colon E\to M$ be a an abstract-definable $\mathcal{C}^p$ vector bundle, and let $U$ be an abstract-definable open subset of $M$. A \textit{local abstract-definable section} of $E$ over $U$ is an abstract-definable map $s\colon U\to E$ satisfying $\pi\circ s=\text{id}_U$. If, in addition, $s$ is $\mathcal{C}^p$, then we say that $s$ is a \textit{local abstract-definable $\mathcal{C}^p$ section}. In the case $U=M$, $s$ is called a \textit{global abstract-definable} (\textit{$\mathcal{C}^p$}) \textit{section}. 
\end{definition}

If $(U,\phi)$ is a chart on $M$, the maps $\partial/\partial x^i\colon U\to TM$, given by $x\mapsto \partial/\partial x^i|_x$ ($i=1,\ldots,m$), are abstract-definable $\mathcal{C}^p$ sections of $TM$ over $U$. Similarly, each $dx^i\colon U\to T^*M$ defined as $x\mapsto dx^i|_x$ is an abstract-definable $\mathcal{C}^p$ section of the cotangent bundle $T^*M$ over $U$. 

\begin{lemma}\label{3vb}
Let $s$ and $t$ be abstract-definable sections of an abstract-definable $\mathcal{C}^p$ vector bundle $\pi\colon E\to M$ over an abstract-definable open set $U\subseteq M$, and let $f\colon U\to R$ be an abstract-definable function. Then the sum $s+t$ and product $fs$ defined respectively by $(s+t)(x)\mathrel{\mathop:}= s(x)+t(x)$ and $(fs)(x)\mathrel{\mathop:}= f(x)\cdot s(x)$ are abstract-definable sections of $E$ over $U$. If in addition $s$, $t$ and $f$ are $\mathcal{C}^p$, then so are $s+t$ and $fs$.
\end{lemma}
\begin{proof}
Lemma 5.3 (\cite{rf2017}, p. 50).
\end{proof}

\begin{definition}\label{4vb}
A \textit{local absctract-definable frame} for an abstract-definable $\mathcal{C}^p$ vector bundle $\pi\colon E\to M$ of rank $d$ is a $d$-tuple of local sections $(s_1,\ldots,s_d)$ of $E$ over an abstract-definable open subset $U\subseteq M$ such that at each point $x\in U$, the elements $s_1(x),\ldots,s_d(x)$ form a basis for the fiber $E_x$. If, in addition, the sections $s_1,\ldots,s_d$ are $\mathcal{C}^p$, then $(s_1,\ldots,s_d)$ is called a \textit{local abstract-definable $\mathcal{C}^p$ frame} for $E$ over $U$. In the case $U=M$, the $d$-tuple is said to be a \textit{global abstract-definable} (\textit{$\mathcal{C}^p$}) \textit{frame}.
\end{definition}

For any chart $(U,\phi)$ on $M$, $(\partial/\partial x^1, \ldots, \partial/\partial x^m)$ is a local abstract-definable $\mathcal{C}^p$ frame for the tangent bundle $TM$ as well as $(dx^1,\ldots,dx^m)$ forms a local abstract-definable $\mathcal{C}^p$ frame, the \textit{coordinate frame}, for the cotangent bundle $T^*M$.

Suppose $\pi\colon E\to M$ is an abstract-definable $\mathcal{C}^p$ vector bundle of rank $d$ and $\varphi\colon \pi^{-1}(\Omega)\to \Omega\times R^d$ is a local trivialization of $E$. Let $t_1,\ldots,t_d\colon \Omega\to E$ be abstract-definable $\mathcal{C}^p$ maps given by the rule $t_i(x)\mathrel{\mathop:}=(\varphi^{-1}\circ \widetilde{e}_i)(x)$, where $\widetilde{e}_i\colon \Omega\to \Omega\times R^d$ is the abstract-definable $\mathcal{C}^p$ map $x\mapsto (x,e_i)$ and $\{e_j\}_j$ denotes the standard basis for $R^d$. Then $(t_1,\ldots,t_d)$ is an abstract-definable $\mathcal{C}^p$ frame for $E$ over $\Omega$ and is called the local abstract-definable $\mathcal{C}^p$ frame \textit{associated with $\varphi$}.

\begin{lemma}\label{5vb}
Let $\varphi\colon \pi^{-1}(\Omega)\to \Omega\times R^d$ be a trivialization of an abstract-definable $\mathcal{C}^p$ vector bundle $\pi\colon E\to M$, and $(t_1,\ldots,t_d)$ the local abstract-definable $\mathcal{C}^p$ frame associated with $\varphi$. Then, a map $s\mathrel{\mathop:}= \sum_{i=1}^d b^i t_i$, where $b^i$ are $R$-valued functions on $\Omega$, is an abstract-definable $\mathcal{C}^p$ section of $E$ over $\Omega$ if, and only if, its coefficients $b^i$ relative to the frame $(t_1,\ldots,t_d)$ are abstract-definable $\mathcal{C}^p$.  
\end{lemma}
\begin{proof}
Lemma 5.5 (\cite{rf2017}, p. 52).
\end{proof}

The theorem below is an extension of Lemma \ref{5vb} in the sense that in a similar fashion it characterizes abstract-definable $\mathcal{C}^p$ sections of an abstract-definable $\mathcal{C}^p$ vector bundle $\pi\colon E\to M$ over any abstract-definable open subset of $M$, unlike in Lemma \ref{5vb} where the corresponding sets are elements of a trivializing open cover of $M$. Theorem \ref{6vb} plays fundamental role in allowing us to give a local description of the abstract-definable analogues of global smooth differential forms, examined in the next section.

\begin{theorem}\label{6vb}
Let $\pi\colon E\to M$ be an abstract-definable $\mathcal{C}^p$ vector bundle of rank $d$, and $U$ an abstract-definable open subset of $M$. Suppose $(s_1,\ldots,s_d)$ is a $\mathcal{C}^p$ frame for $E$ over $U$. Then the map $s\mathrel{\mathop:}= \sum_{i=1}^d c^i s_i$, where $c^i$ are $R$-valued functions on $U$, is an abstract-definable $\mathcal{C}^p$ section of $E$ over $U$ if and only if the coefficients $c^i$ are abstract-definable $\mathcal{C}^p$.
\end{theorem}
\begin{proof}
Assume that $s$ is an abstract-definable $\mathcal{C}^p$ section of $E$ over $U$. At any point in $U$ there is a chart $\Omega\subseteq U$ on $M$ which is also a trivializing open set for $E$, with $\varphi$ as its corresponding trivialization. Let $(t_1,\ldots,t_d)$ be the local abstract-definable $\mathcal{C}^p$ frame associated with $\varphi$. If $(b^i)_i$ and $(a^i_j)_i$ are the coefficients of $s$ and $s_j$ in terms of the frame $(t_1,\ldots,t_d)$ respectively, then in view of Lemma \ref{5vb} they are abstract-definable $\mathcal{C}^p$ functions. From $\sum_i b^it_i=\sum_{i,j}c^ia^i_jt_i$ we have the matrix equality $[b^i]=[a_j^i][c^j]$, with $[a_j^i]$ invertible. By Cramer's rule, each $c^i$ is abstract-definable $\mathcal{C}^p$ on $\Omega$. The ``if'' direction is just a straightforward application of Lemma \ref{3vb}.
\end{proof}

\section{Abstract-definable forms}

\begin{definition}\label{1f}
An \textit{abstract-definable} (\textit{$\mathcal{C}^p$}) \textit{1-form} on $M$ is an abstract-definable ($\mathcal{C}^p$) section $\omega$ of the cotangent bundle $\pi\colon T^*M\to M$.
\end{definition}

In order to give a characterization of abstract-definable $\mathcal{C}^p$ $1$-forms on $M$ in terms of the coordinate frames, we restate Theorem \ref{6vb} for $E=T^*M$ and $s_i=dx^i$.

\begin{lemma}\label{2f}
Let $(U,\phi)$ be a chart on $M$. The map $\omega\mathrel{\mathop:}= \sum_{i=1}^m \omega_i dx^i$, where $\omega_i$ are $R$-valued functions, is an abstract-definable $\mathcal{C}^p$ 1-form on $U$ if and only if the coefficients $\omega_i$ are abstract-definable $\mathcal{C}^p$. 
\end{lemma}

\begin{theorem}\label{3f}
Let $\omega\colon M\to T^* M$ be a map that satisfies the equality $\pi\circ \omega=\text{id}_M$. The following are equivalent: 
\begin{enumerate}
\item[(i)] $\omega$ is an abstract-definable $\mathcal{C}^p$ 1-form on $M$.
\item[(ii)] For any chart $(U,\phi)$ on $M$, the map $\omega$ restricted to $U$ is given by $x\mapsto \sum_{i=1}^m \omega_i(x) dx^i|_x$, where the functions $\omega_i\colon U\to R$ are abstract-definable $\mathcal{C}^p$.
\item[(iii)] For any point $x\in M$ there is a chart $(U,\phi)$ on $M$ at $x$ such that the restriction $\omega|_U$ is given by $z\mapsto \sum_{i=1}^m \omega_i(z) dx^i|_z$, where the functions $\omega_i\colon U\to R$ are abstract-definable $\mathcal{C}^p$.
\end{enumerate}
\end{theorem}
\begin{proof}
(i)$\Rightarrow$(ii) For any chart $(U,\phi)$, the restriction of $\omega$ to $U$ can be written as $\sum_{i=1}^m\omega_idx^i$, with $\omega_i$ a function on $U$. Assuming (i), $\omega|_U$ is an abstract-definable $\mathcal{C}^p$ section of $T^*M$ over $U$, and therefore as a consequence of Lemma \ref{2f} the functions $\omega_i$ are abstract-definable $\mathcal{C}^p$.

(ii)$\Rightarrow$(iii) Straightforward.

(iii)$\Rightarrow$(i) Let $x$ be a point in $M$, and by virtue of (iii) let $(U,\phi)$ be a chart on $M$ at $x$ on which $\omega$ is written as $\sum_{i=1}^m\omega_idx^i$, where $\omega_i$ is an abstract-definable $\mathcal{C}^p$ function on $U$. Consider $\widetilde{\phi}\colon T^*U\to \phi(U)\times R^m$ the induced chart on $T^*M$ by $\phi$, and note that $\omega(x)\in T^*U$. To obtain (i) it suffices to conclude, according to Lemma \ref{2f*}, that $\widetilde{\phi}\circ \omega\circ \phi^{-1}$ restricted to $\phi(U\cap \omega^{-1}(T^*U))$ is definable and is extended by a definable $\mathcal{C}^p$ map defined on a definable open subset of $R^m$. But this holds since $\widetilde{\phi}\circ \omega\circ \phi^{-1}|_ {\phi(U\cap \omega^{-1}(T^*U))}$ equals the restriction of the definable $\mathcal{C}^p$ map $(\text{id}_{\phi(U)},\omega_1\circ \phi^{-1},\ldots,\omega_m\circ \phi^{-1})$ to $\phi(U\cap \omega^{-1}(T^*U))$, which is a definable set in view of the assumption $\pi\circ \omega=\text{id}_M$ and Lemma \ref{2f}.
\end{proof}

If $f\colon M\to R$ is an abstract-definable $\mathcal{C}^p$ function with $p\geq 2$, the \textit{differential} $df$ of $f$ defined as $x\mapsto d_x f$ is an abstract-definable $\mathcal{C}^{p-1}$ $1$-form on $M$. Given a chart $(U,\phi)$ on $M$, the differential of $f$ on $U$ has the well known expression 
$$
df=\sum\limits_{i=1}^m\frac{\partial f}{\partial x^i}dx^i,
$$
where $(\partial f/\partial x^i) (x)\mathrel{\mathop:}=(\partial (f\circ \phi^{-1})/\partial r^i)(\phi(x))$, $x\in U$.

Let $F\colon N\to M$ be an abstract-definable $\mathcal{C}^p$ map. For any abstract-definable $\mathcal{C}^p$ function $g\colon M\to R$, we define the \textit{pullback} of $g$ by $F$ to be the composition $F^*g\mathrel{\mathop:}= g\circ F$, which is an abstract-definable $\mathcal{C}^p$ function on $N$. Now, consider a map $\omega$ from $M$ to $T^* M$ with $\pi\circ \omega=\text{id}_M$. The \textit{pullback} of $\omega$ by $F$ is the map $F^*\omega\colon N\to T^* N$ given by $x\mapsto (F^*\omega)_x$, where $(F^*\omega)_x\colon T_x N\to R$ is the linear function $v\mapsto \omega_{F(x)}(d_x F (v))$. The set of all such maps $\omega$, together with the pointwise operations, forms an $R$-vector space and an $\mathcal{F}(M)$-module.

The chain rule and the definition of pullback give the following. 

\begin{lemma}\label{4f}
Let $F\colon N\to M$ be an abstract-definable $\mathcal{C}^p$ map, $g\colon M\to R$ an abstract-definable $\mathcal{C}^p$ function, and let $\omega,\tau\colon M\to T^*M$ be maps whose composition of $\pi$ with them gives $\text{id}_M$. Then, 
\begin{enumerate}
\item[(i)] $F^*(dg)=d(F^*g)$, where $dg$, $d(F^*g)$ are the differentials of $g$ and $F^*g$, respectively;
\item[(ii)] $F^*(\omega+\tau)=F^*\omega+F^*\tau$;
\item[(iii)] $F^*(g\omega)=(F^*g)(F^*\omega)=(g\circ F) F^*\omega$.
\end{enumerate}
\end{lemma}

\begin{theorem}\label{5f}
The pullback $F^*\omega$ of an abstract-definable $\mathcal{C}^p$ 1-form $\omega$ on $M$ under an abstract-definable $\mathcal{C}^p$ map $F\colon N\to M$ with $p\geq 2$ is an abstract-definable $\mathcal{C}^{p-1}$ 1-form on $N$.
\end{theorem}
\begin{proof}
Proposition 6.6 (\cite{rf2017}, p. 60).
\end{proof}

Let $k$ be a nonnegative integer. The k\textit{th exterior power of the cotangent bundle} $\bigwedge^kT^* M$ of $M$ is the disjoint union $\bigcup_{x\in M} \{x\}\times \bigwedge^kT^*_x M$, where $\bigwedge^kT^*_x M$ denotes the $R$-vector space of all alternating $k$-linear functions $T_x^*M\times\cdots\times T_x^*M\to R$. Any chart $(U,\phi)$  on $M$ induces a chart $\widetilde{\phi}\colon \bigwedge^kT^* U\to \phi(U)\times R^{\binom{m}{k}}$ on $\bigwedge^kT^*M$ given by
$$
(x,\sum_I a_Idx^I|_x)\mapsto (\phi(x),(a_I)_I),
$$ 
where $I\in \{(i_1,\ldots,i_k)\in \mathbb{N}^k\,:\, 1\leq i_1<\cdots <i_k\leq m\}$, $dx^I|_x\mathrel{\mathop:}= dx^{i_1}|_x \wedge \cdots \wedge dx^{i_k}|_x$, and $\bigwedge^kT^* U\mathrel{\mathop:}= \bigcup_{x\in U} \{x\}\times \bigwedge^kT_x^* U$. This makes the $k$th exterior power of the cotangent bundle of $M$ into an abstract-definable $\mathcal{C}^p$ manifold of dimension $m+\binom{m}{k}$. Moreover, the natural projection $\pi\colon \bigwedge^kT^*M\to M$ is an abstract-definable $\mathcal{C}^p$ vector bundle of rank $\binom{m}{k}$ whose fibers at each $x\in M$ are the vector spaces $\bigwedge^k T_x^*M$.

\begin{definition}\label{6f}
An \textit{abstract-definable} (\textit{$\mathcal{C}^p$}) \textit{$k$-form} on $M$ is an abstract-definable ($\mathcal{C}^p$) section $\omega$ of $\pi\colon \bigwedge^kT^*M\to M$, the $k$th exterior power of the cotangent bundle of $M$.
\end{definition}

Since $\bigwedge^0 T^*M=M\times R$, the vector space of all abstract-definable $\mathcal{C}^p$ $0$-forms equals that of all abstract-definable $\mathcal{C}^p$ functions on $M$. 
\\

\textsc{Notation}. 
For the remainder of the text, $I$ denotes a $k$-tuple $(i_1,\ldots,i_k)$ with $1\leq i_1<\cdots<i_k\leq m$, $dx^I$ is a short for $dx^{i_1}\wedge\cdots \wedge dx^{i_k}$, and ${dx^I|}_x$ designates ${dx^{i_1}|}_x\wedge\cdots\wedge {dx^{i_k}|}_x$. Also, $\pi$ stands for the projection $\bigwedge^kT^*M\to M$.
\\

Note that given a chart $(U,\phi)=(U,x^1,\ldots,x^m)$ on $M$, the $\binom{m}{k}$-tuple $(dx^I)_I$ forms a local $\mathcal{C}^p$ frame for the abstract-definable $\mathcal{C}^p$ vector bundle $\bigwedge^k T^* M$ over $U$.

In the sequel, we restate Lemma \ref{2f} and Theorem \ref{3f} for abstract-definable $\mathcal{C}^p$ $k$-forms. 

\begin{lemma}\label{7f}
Let $(U,\phi)$ be a chart on $M$. The map $\omega\mathrel{\mathop:}= \sum_I\omega_Idx^I$, where $\omega_I$ are $R$-valued functions on $U$, is an abstract-definable $\mathcal{C}^p$ $k$-form on $U$ if and only if the coefficients $\omega_I$ are abstract-definable $\mathcal{C}^p$.
\end{lemma}
\begin{proof}
Lemma 6.12 (\cite{rf2017}, p. 64).
\end{proof}

By following the proof of Theorem \ref{3f}, we obtain the subsequent characterizations of the abstract-definable $\mathcal{C}^p$ $k$-forms.

\begin{theorem}\label{8f}
Let $\omega\colon M\to \bigwedge^kT^*M$ that satisfies the equality $\pi\circ \omega=\text{id}_M$. The following are equivalent.
\begin{enumerate}
\item[(i)] $\omega$ is an abstract-definable $\mathcal{C}^p$ $k$-form on $M$.
\item[(ii)] For any chart $(U,\phi)$ on $M$, the map $\omega$ restricted to $U$ is given by $x\mapsto \sum_{I}\omega_I(x){dx^I|}_x$, where the functions $\omega_I\colon U\to R$ are abstract-definable $\mathcal{C}^p$.
\item[(iii)] For any point $x\in M$ there is a chart $(U,\phi)$ on $M$ at $x$ such that the restriction ${\omega|}_U$ is given by $z\mapsto \sum_{I}\omega_I(z){dx^I|}_z$, where the functions $\omega_I\colon U\to R$ are abstract-definable $\mathcal{C}^p$.
\end{enumerate}
\end{theorem}

Let $\omega,\tau$ be maps $x\mapsto (x,\omega_x)\colon M\to \bigwedge^kT^*M$ and $x\mapsto (x,\tau_x)\colon M\to \bigwedge^l T^*M$, respectively. The \textit{wedge product} of $\omega$ and $\tau$ is the map $\omega\wedge \tau\colon M\to \bigwedge^{k+l} T^*M$ which associates to each $x\in M$ the element $\omega_x\wedge \tau_x \in \bigwedge^{k+l}T_x^*M$.

\begin{theorem}\label{9f}
If $\omega$ is an abstract-definable $\mathcal{C}^p$ $k$-form on $M$ and $\tau$ is an abstract-definable $\mathcal{C}^q$ $l$-form on $M$, then $\omega\wedge \tau$ is an abstract-definable $\mathcal{C}^r$ $(k+l)$-form on $M$ with $r\mathrel{\mathop :}= \min\{p,q\}$.
\end{theorem}
\begin{proof}
Proposition 6.14 (\cite{rf2017}, p. 65).
\end{proof}

\begin{corollary}\label{10f}
Let $f_i\colon M\to R$ be abstract-definable $\mathcal{C}^{p_i}$ functions with $p_i\geq 2$, $i=1,\ldots,k$. Then the wedge product of their differentials $df_1\wedge \cdots \wedge df_k$ is an abstract-definable $\mathcal{C}^p$ $k$-form on $M$, where $p$ is the least of $p_1-1,\ldots,p_k-1$. Moreover, for any chart $(U,\phi)$ on $M$ 
$$
df_1\wedge \cdots \wedge df_k= \sum\limits_I\frac{\partial (f_1,\ldots,f_k)}{\partial (x^{i_1},\ldots,x^{i_k})} dx^I\ \text{on}\ U,
$$
and  
$$
\frac{\partial (f_1,\ldots,f_k)}{\partial (x^{i_1},\ldots,x^{i_k})}(x)\mathrel{\mathop:}= \det
\begin{bmatrix}
\frac{\partial f_1}{\partial x^{i_1}}(x) &\cdots &\frac{\partial f_1}{\partial x^{i_k}}(x)\\
\vdots    &\ddots   &\vdots\\
\frac{\partial f_k}{\partial x^{i_1}}(x) &\cdots &\frac{\partial f_k}{\partial x^{i_k}}(x)
\end{bmatrix}.
$$
\end{corollary}
\begin{proof}
The first part of the corollary follows immediately from Theorem \ref{9f}. See Corollary 6.15 (\cite{rf2017}, p. 66) for the complete proof.
\end{proof}

If $(U,\phi)=(U,x^1,\ldots,x^m)$ and $(V,\psi)=(V,y^1,\ldots,y^m)$ are two overlapping charts on $M$ and $p\geq 2$ (recall that $M$ is an abstract-definable $\mathcal{C}^p$ manifold), then 
$$
dy^{j_1}\wedge\cdots\wedge dy^{j_k}=\sum\limits_{1\leq i_1<\ldots<i_k\leq m}\frac{\partial (y^{j_1}, \ldots,y^{j_k})}{\partial (x^{i_1},\ldots,x^{i_k})} dx^{i_1}\wedge\cdots\wedge dx^{i_k},
$$
in view of Corollary \ref{10f}.

Consider an abstract-definable $\mathcal{C}^p$ map $F\colon N\to M$ and a map $\omega\colon M\to \bigwedge^kT^*M$ given by $x\mapsto (x,\omega_x)$. The \textit{pullback of $\omega$ by $F$} is the map $F^*\omega\colon N\to \bigwedge^k T^*N$ defined as $x\mapsto (F^*\omega)_x$,
where $(F^*\omega)_x\colon T_xN\times \cdots \times T_xN\to R$ is the following $k$-linear function 
$$
(F^*\omega)_x (v_1,\ldots,v_k)\mathrel{\mathop:}= \omega_{F(x)}(d_x F v_1,\ldots,d_x F v_k).
$$ 
The set of all such maps $\omega$, equipped with pointwise operations, forms an $R$-vector space and a $\mathcal{F}(M)$-module, where $\mathcal{F}(M)$ denotes the ring of all $R$-valued functions on $M$.

The chain rule and the definition of pullback yield the following.

\begin{lemma}\label{11f}
Let $F\colon N\to M$ be an abstract-definable $\mathcal{C}^p$ map, $g\colon M\to R$ an abstract-definable $\mathcal{C}^p$ function, and let $\omega,\tau$ be the maps $x\mapsto (x,\omega_x)\colon M\to \bigwedge^kT^*M$ and $x\mapsto(x,\tau_x)\colon M\to \bigwedge^lT^*M$, respectively. Then, 
\begin{enumerate}
\item[(i)] $F^*(\omega\wedge \tau)=(F^*\omega)\wedge (F^*\tau)$.
\end{enumerate} 
In the case of $k=l$, 
\begin{enumerate}
\item[(ii)] $F^*(\omega+\tau)= F^*\omega+ F^*\tau$;
\item[(iii)] $F^*(g\omega)=  (F^*g)(F^*\omega)$. Particularly, when $g\equiv a\in R$ is a constant function on $M$, $F^*(a\omega)=a(F^*\omega)$.
\end{enumerate}
\end{lemma}

\begin{theorem}\label{12f}
The pullback $F^*\omega$ of an abstract-definable $\mathcal{C}^p$ $k$-form $\omega$ on $M$ under an abstract-definable $\mathcal{C}^p$ map $F\colon N\to M$ with $p\geq 2$ is an abstract-definable $\mathcal{C}^{p-1}$ $k$-form on $N$.
\end{theorem}
\begin{proof}
Proposition 6.19 (\cite{rf2017}, p. 69).
\end{proof}

\section{Exterior derivative}

As we have seen in the latter section, the classes of differentiability  $\mathcal{C}^p$ for $1\leq p<\infty$ are not closed under differentiation. In walking the path towards a de Rham-like cohomology theory for o-minimal manifolds, there was no need up to now of this closure condition, and therefore we were allowed to work within the general setting of an o-minimal expansion of a real closed field.
However, since we aim to construct cochain complexes whose objects are sets of abstract-definable forms by following the lines of the classical de Rham cohomology theory, we must establish exterior derivative, and this requires such a closure condition on differentiability. So, we turn our attention to abstract-definable $\mathcal{C}^\infty$ manifolds, and by virtue of \cite{legalrolin2009} we need to restrict ourselves to an o-minimal expansion of the real field which possesses $\mathcal{C}^\infty$ cell decomposition. Recall that in building up abstract-definable $\mathcal{C}^p$ partitions of unity we made heavy use of results, specifically Theorem \ref{2pu} by A. Thamrongthanyalak and Lemma \ref{3pu}, in the definable context that ``a priori'' only hold for $p<\infty$. Howbeit, if in addition to admitting a smooth cell decomposition the o-minimal expansion of the real field defines the exponential function, then we have in hand analogous results (Theorem 1.1, \cite{fischer2008}, p. 497) and (Corollary 1.2, \cite{fischer2008}, p. 497), respectively. Therefore, proceeding just like in Theorem \ref{4pu} (which by the way, it was fully inspired by Lemma 4.6, \cite{fischer2008}) where Lemma \ref{3pu} is replaced with Corollary 1.2 (\cite{fischer2008}), we obtain abstract-definable $\mathcal{C}^\infty$ partitions of unity\index{abstract-definable! @$\mathcal{C}^\infty$ partition of unity}, and consequently abstract-definable $\mathcal{C}^\infty$ bump functions as in Corollary \ref{8pu}.
\\

\textit{In view of this, we fix from now on an o-minimal expansion $\mathcal{R}$ of the real field $\mathbb{R}$ that admits smooth cell decomposition and defines the exponential function.  By ``definable'' we mean ``definable in $\mathcal{R}$ with parameters in $\mathbb{R}$''}.
\\

In addition to \textit{all we developed so far for abstract-definable $\mathcal{C}^p$ manifolds with $1\leq p<\infty$ holding for the case $p=\infty$}, we have clear improvements like the pullback $F^*\omega$ of an abstract-definable $\mathcal{C}^\infty$ $k$-form $\omega$ on an abstract-definable $\mathcal{C}^\infty$ manifold $M$ under an abstract-definable $\mathcal{C}^\infty$ map $F\colon N\to M$ is an abstract-definable $\mathcal{C}^\infty$ $k$-form on the abstract-definable $\mathcal{C}^\infty$ manifold $N$, that is, there is no decreasing in the differentiability class of $F^*\omega$.
\\

\textsc{Notation}. For the remainder of the text, unless otherwise stated, $(M,\mathcal{A})$ and $(N,\mathcal{B})$ denote abstract-definable $\mathcal{C}^\infty$ manifolds of dimensions $m$ and $n$, respectively. 
\\

For each $k$, let $\Omega^k(M)$ denote the set of all abstract-definable $\mathcal{C}^\infty$ $k$-forms\index{abstract-definable! @$\mathcal{C}^\infty$ $k$-forms}. This set equipped with the pointwise sum and scalar multiplication of maps forms an $\mathbb{R}$-vector space. Take $\Omega^*(M)$ to be the $\mathbb{R}$-vector space given by the the direct sum
$$
\Omega^*(M)\mathrel{\mathop:}= \bigoplus\limits_{j=0}^m \Omega^k(M).
$$
With the wedge product, the vector space $\Omega^*(M)$ becomes an anticommutative graded algebra, where the grading is the degree of the abstract-definable $\mathcal{C}^k$ forms on $M$. Also, from Lemma \ref{11f}, if $F\colon M\to N$ is an abstract-definable $\mathcal{C}^\infty$ map then the \textit{pullback map} $F^*\colon \Omega^*(N)\to \Omega^*(M)$ is a homomorphism of graded algebras.

An \textit{exterior derivative} on $M$ is an $\mathbb{R}$-linear map $D\colon \Omega^*(M)\to \Omega^*(M)$ satisfying the conditions: \begin{enumerate}
\item[(i)] $D$ is an antiderivation of degree 1, that is, $D\omega\in \Omega^{k+1}(M)$ and $D(\omega\wedge \tau)=D\omega\wedge \tau+(-1)^k\omega\wedge D\tau$, for $\omega\in \Omega^k(M)$, $\tau\in \Omega^l(M)$;
\item[(ii)] $D\circ D=0$;
\item[(iii)] for any abstract-definable $\mathcal{C}^p$ function $f\colon M\to R$, $Df$ equals the differential $df$ of $f$.
\end{enumerate}  

An $\mathbb{R}$-linear operator $D\colon \Omega^*(M)\to \Omega^*(M)$ is called \textit{local} if has the property that for all $k\geq 0$, if $\omega$ is an abstract-definable $\mathcal{C}^\infty$ $k$-form on $M$ in such a way that $\omega|_U=0$ for some abstract-definable open subset $U$ of $M$, then $D\omega=0$ on $U$; or equivalently, for all $k\geq 0$ and for every two abstract-definable $\mathcal{C}^\infty$ $k$-forms $\omega,\tau\in \Omega^k(M)$ agreeing on an abstract-definable open subset $U\subseteq M$, we have $D\omega=D\tau$ on $U$.

\begin{theorem}\label{1ed}
Every antiderivation $D$ on $\Omega^*(M)$ is a local operator.
\end{theorem}
\begin{proof}
Proposition 7.6 (\cite{rf2017}, p. 74).
\end{proof}

\begin{lemma}\label{2ed}
Let $(U,x^1,\ldots,x^m)$ be a chart on $M$. There exists a unique exterior derivative $d_U$ on $U$.
\end{lemma}
\begin{proof}
For any $\tau\mathrel{\mathop:}=\sum_I \tau_Idx^I\in \Omega^k(U)$, with $\tau_I$ an abstract-definable $\mathcal{C}^\infty$ function on $U$, define a map $d_U\colon \Omega^*(U)\to \Omega^*(U)$ in such a way that its restriction to each $\Omega^k(U)$ is given by
$$
d_U(\tau)\mathrel{\mathop:}= \sum\limits_I(\sum\limits_j\frac{\partial \tau_I}{\partial x^j}dx^j)\wedge dx^I.
$$
It is not hard to see that $d_U$ is an exterior derivative on $U$. Moreover, the properties (i)-(iii) imply the uniqueness of $d_U$.
\end{proof}

\begin{lemma}\label{2ed*}
Let $(U,\phi)$ be a chart on $M$, and $\omega\in \Omega^k(U)$. There is $\widetilde{\omega}\in \Omega^k(M)$ and an abstract-definable open set $V\subseteq U$ such that $\widetilde{\omega}=\omega$ on $V$.
\end{lemma}
\begin{proof}
Write $\omega$ as $\sum_I\omega_Idx^I$, where $\omega_I$ are  abstract-definable $\mathcal{C}^\infty$ functions on $U$. Consider an abstract-definable $\mathcal{C}^\infty$ bump function $\rho\colon M\to \mathbb{R}$ supported in $U$, and let $V\subseteq U$ be the abstract-definable open set on which $\rho$ is identically $1$. Then, defining $\widetilde{\omega}_I\colon M\to \mathbb{R}$ as $\rho\cdot \omega_I$ on $U$, and $0$ on $M\setminus U$, it follows that $\omega_I$ is an abstract-definable $\mathcal{C}^\infty$ function, and on $V$ both of functions $\widetilde{\omega}_I$ and $\omega_I$ coincide. Similarly, we obtain abstract-definable $\mathcal{C}^\infty$ functions $\widetilde{x}^i\colon M\to \mathbb{R}$ extending $x^i|_V$. Now, set $
\widetilde{\omega}\mathrel{\mathop:}= \sum_I\widetilde{\omega}_Id\widetilde{x}^I$. By Corollary \ref{10f} and the fact that $\Omega^k(M)$ is a module over the ring of all abstract-definable $\mathcal{C}^\infty$ functions on $M$, $\widetilde{\omega}\in \Omega^k(M)$. Finally, note that $\widetilde{\omega}|_V=\sum_I{\widetilde{\omega}_I|}_V {d\widetilde{x}^{i_1}|}_V\wedge \cdots \wedge {d\widetilde{x}^{i_k}|}_V$.
\end{proof}

\begin{theorem}\label{3ed}
There exists an exterior derivative $d\colon \Omega^*(M)\to \Omega^*(M)$ which is uniquely determined by the conditions (i)-(iii) above.
\end{theorem}
\begin{proof}
For each $k$, define $d_{(k)}\colon \Omega^k(M)\to \Omega^{k+1}(M)$ to be the linear map which associates an abstract-definable $\mathcal{C}^\infty$ $k$-form $\omega$ on $M$ to the map 
$$
x\mapsto (d_U\omega)_x,
$$ 
where $U$ is a chart on $M$ at $x$ and $d_U$ is given as in the proof of Lemma \ref{2ed}. The fact that $d_{(k)}$ does not depend on the choice of the chart $U$ follows from Lemma \ref{2ed}. Now, take $d\colon \Omega^*(M)\to \Omega^*(M)$ to be the linear map given by 
$$
d(\omega_0+\cdots+\omega_m)\mathrel{\mathop:}=d_{(0)}(\omega_0)+\cdots+d_{(m)}(\omega_m),
$$ 
with $\omega_k\in \Omega^k(M)$. Such a map satisfies the conditions (i)-(iii), since each $d_U$ does so. The uniqueness of the exterior derivative $d$ is obtained from the fact that for any exterior derivative $D\colon \Omega^*(M)\to \Omega^*(M)$ and abstract-definable $\mathcal{C}^\infty$ functions $f_1,\ldots,f_k\colon M\to \mathbb{R}$, $D(df_1\wedge \cdots \wedge df_k)=0$, and from Lemma \ref{2ed*}.
\end{proof}

Lemma \ref{4f}(i) and Theorem \ref{9f} give the following.

\begin{theorem}\label{4ed}
Consider an abstract-definable $\mathcal{C}^\infty$ map $F\colon N\to M$, and let $\omega\in \Omega^k(M)$. Then, $d(F^*\omega)=F^*d\omega$.
\end{theorem}

If in Theorem \ref{4ed} we replace $N$ with an abstract-definable open subset $U\subseteq M$, and $F$ with the inclusion $\imath\colon U\to M$, we effortlessly obtain the following.

\begin{corollary}\label{5ed}
Let $U$ be an abstract-definable open subset of $M$, and $\omega\in \Omega^k(M)$. Then, ${(d\omega)|}_U=d\left({\omega|}_U\right)$.
\end{corollary}

\section{O-minimal de Rham cohomology}

An abstract-definable $\mathcal{C}^\infty$ $k$-form $\omega$ on $M$ is said to be \textit{closed} if its derivative vanishes, that is, $d\omega=0$, and \textit{exact} if there is an abstract-definable $\mathcal{C}^\infty$ ($k-1$)-form $\tau$ on $M$ such that $\omega=d\tau$. Observe that every exact abstract-definable $\mathcal{C}^\infty$ $k$-form on $M$ is closed, since $d^2=0$. We denote by $Z^k(M)$ the $\mathbb{R}$-vector space of all closed abstract-definable $\mathcal{C}^\infty$ $k$-forms on $M$, and by $B^k(M)$ the $\mathbb{R}$-vector space of all exact abstract-definable $\mathcal{C}^\infty$ $k$-forms on $M$. Hence, $B^k(M)$ is a subspace of $Z^k(M)$ and we may form the quotient vector space $Z^k(M)/B^k(M)$. 

\begin{definition}\label{1dr}
The $\mathbb{R}$-vector space 
$$
H^k(M)\mathrel{\mathop:}= Z^k(M)/B^k(M)
$$
is called the k\textit{th o-minimal de Rham cohomology group} of $M$.
\end{definition}

Recall from Theorem \ref{4} that $M$ has finitely many definably connected components.

\begin{theorem}\label{2dr}
Let $d$ be the number of definably connected components of $M$. Then, $H^0(M)= \mathbb{R}^d$.
\end{theorem}
\begin{proof}
Note that the vector space of all abstract-definable functions on $M$ which are constant on each definably connected component of $M$ is $d$-dimensional. Also, such a vector space agrees with the one constituted of all locally constant abstract-definable functions on $M$. This in turn coincides with $Z^0(M)$, since for any chart $(U,\phi)$ on $M$, $df=0$ on $U$ implies that $f$ is constant on each element of a finite partition of $U$ into open sets; and conversely, if $f$ is locally constant then each point in $U$ has a neighborhood in which $df$ vanishes, hence $df=0$ on $U$, and from the arbitrariness of $U$ it follows that $df$ is the identically zero map in $\Omega^1(M)$.
\end{proof}

Because $\Omega^k(M)=0$ for each $k> m$, we immediately get 

\begin{theorem}\label{3dr}
$H^k(M)= 0$, for all $k>m$.
\end{theorem}

In view of Theorems \ref{2dr} and \ref{3dr} above, $H^0(\mathbb{R})=\mathbb{R}$ and $H^k(\mathbb{R})=0$ for $k\geq 2$. If we put $\mathcal{R}\mathrel{\mathop:}=\mathbb{R}_{\exp}$, on the other hand, it follows that $H^1(\mathbb{R})=\mathbb{R}$. Indeed, $[1/(1+x^2)]dx\in \Omega^1(\mathbb{R})=Z^1(\mathbb{R})$ has primitive $\arctan(x)$, which is not definable in $\mathbb{R}_{\exp}$ by virtue of Theorem 1 (\cite{bianconi1997}) and some trigonometric identities. In other words, $B^1(\mathbb{R})$ is a proper vector subspace of $Z^1(\mathbb{R})$, thereby $H^1(\mathbb{R})\neq 0$. Recalling that the classical de Rham cohomology groups $H_{\text{dR}}^k(\mathbb{R})$ of $\mathbb{R}$ are all trivial, we conclude that the o-minimal de Rham cohomology does not necessarily agree with the classical one. 

As mentioned above, any abstract-definable $\mathcal{C}^\infty$ map $F\colon N\to M$ induces a homomorphism $F^*\colon \Omega^*(M)\to \Omega^*(N)$ of graded algebras, the \textit{pullback map}, which preserves closed and exact abstract-definable forms, \textit{i.e.}, $F^*(Z^k(M))\subseteq Z^k(N)$ and $F^*(B^k(M))\subseteq B^k(N)$. The map $F^*$ in turn induces a map $F^\sharp\colon H^k(M)\to H^k(N)$, the \textit{pullback map in cohomology}, by setting 
$$
F^\sharp([\omega])\mathrel{\mathop:}= [F^*\omega].
$$
The linearity of $F^*$ implies that of $F^\sharp$. Moreover, if $\text{id}_M\colon M\to M$ is the identity map, then so is $\text{id}_M^\sharp\colon H^k(M)\to H^k(M)$; and for any abstract-definable $\mathcal{C}^\infty$ maps $F\colon N\to M$ and $G\colon M\to P$, we have $(G\circ F)^\sharp=F^\sharp\circ G^\sharp$. In other words, $^\sharp$ is a contravariant functor from the category of abstract-definable $\mathcal{C}^\infty$ manifolds and abstract-definable $\mathcal{C}^\infty$ maps to the category of vector spaces over $\mathbb{R}$. This proves the following.

\begin{theorem}\label{3,5dr}
$F^\sharp\colon H^k(M)\to H^k(N)$ is a linear isomorphism whenever $F\colon N\to M$ is an abstract-definable $\mathcal{C}^\infty$ diffeomorphism.
\end{theorem}

Given $[\omega],[\tau]\in H^k(M)$, we define 
$$
[\omega]\wedge [\tau]\mathrel{\mathop:}= [\omega\wedge \tau]. 
$$
Therefore, equipping the $\mathbb{R}$-vector space
$$
H^*(M)\mathrel{\mathop:}= \bigoplus\limits_{k=0}^m H^k(M)
$$
with this product, $H^*(M)$ becomes a graded algebra over $\mathbb{R}$. The anticommutativity of $H^*(M)$ is inherited from that of $\Omega^*(M)$. 

If $V,W\subseteq M$ are abstract-definable open subsets whose union covers $M$, then we have four inclusion maps: $\imath_V\colon V\to M$, $\imath_W\colon W\to M$, $\jmath_V\colon V\cap W\to V$, and $\jmath_W\colon V\cap W\to W$. Note that the pullback map $\imath_V^*\colon \Omega^k(M)\to \Omega^k(V)$ is the map that restricts the domain of an abstract-definable $\mathcal{C}^\infty$ $k$-form on $M$ to $V$.

\begin{lemma}\label{4dr}
Let $V, W\subseteq M$ be abstract-definable open cover of $M$. For each $k\geq 0$, the sequence below is exact
$$
0\to \Omega^k(M)\stackrel{\imath_k}{\to} \Omega^k(V)\oplus \Omega^k(W)\stackrel{\jmath_k}{\to} \Omega^k(V\cap W)\to 0,
$$
where $\imath_k\colon \Omega^k(M)\to \Omega^k(V)\oplus \Omega^k(W)$ is the map given by 
\begin{equation}
\omega\mapsto (\imath^*_V\omega, \imath^*_W\omega)=(\omega|_V,\omega|_W)
\end{equation}
whereas $\jmath_k\colon \Omega^k(V)\oplus \Omega^k(W)\to \Omega^k(V\cap W)$ is defined as 
\begin{equation}
(\omega,\tau)\mapsto \jmath_V^*\omega-\jmath^*_W\tau=\omega|_{V\cap W}-\tau|_{V\cap W}.
\end{equation}
In the case $V\cap W$ is empty, we have $\Omega^k(V\cap W)=0$, and consequently $\jmath_k$ is the zero map.
\end{lemma}
\begin{proof}
We must prove the following statements, for each $k\geq 0$: (i) $\imath_k$ is one-to-one; (ii) $\ker(\jmath_k)=\text{im}(\imath_k)$; and (iii) $\jmath_k$ is onto.

(i) If $\imath_k(\omega)=(0,0)$, then $0=\imath^*_V(\omega)=\omega|_V$ and $0=\imath^*_W(\omega)=\omega|_W$, and since the sets $V$ and $W$ cover $M$, it results that $\omega=0$.

(ii) Let $(\omega,\tau)\in \Omega^k(V)\oplus \Omega^k(W)$ with $\jmath^*_V\omega-\jmath^*_W\tau=0$. This means that $\omega$ and $\tau$ agree on $V\cap W$. As a consequece, the map $\sigma\colon M\to \bigwedge^kT^*M$ given by \[
\sigma\mathrel{\mathop:}= \begin{cases}
\omega   &\text{on }\ V\\
\tau     &\text{on }\ W
\end{cases}
\]
is an abstract-definable $\mathcal{C}^\infty$ $k$-form on $M$ satisfying $\imath_k(\sigma)=(\omega,\tau)$. Conversely, for any abstract-definable $\mathcal{C}^\infty$ $k$-form $\omega$ on $M$, we have $\jmath_k(\imath_k(\omega))=\jmath_k^*(\omega|_V,\omega|_W)=\omega|_{V\cap W}-\omega|_{V\cap W}=0$.

(iii) Let $\omega\in \Omega^k(V\cap W)$. By the abstract-definable smooth version of Proposition \ref{5pu}, there are abstract-definable $\mathcal{C}^\infty$ nonnegative functions $f_V,f_W\colon M\to R$ such that $f_V+f_W=1$, $\text{supp}(f_V)\subseteq V$ and $\text{supp}(f_W)\subseteq W$. Take $\sigma_1\in \Omega^k(V)$ and $\sigma_2\in \Omega^k(W)$ to be, respectively, the maps
\[\sigma_1\mathrel{\mathop:}=
\begin{cases}
f_W\cdot \omega  &\text{on }\ V\cap W\\
0    &\text{on }\ V\setminus (V\cap W)
\end{cases}
\]
and
\[\sigma_2\mathrel{\mathop:}=
\begin{cases}
-f_V\cdot \omega &\text{on }\ V\cap W\\
0  &\text{on }\ W\setminus (V\cap W)
\end{cases}.
\]
Hence, 
\begin{align*}
\jmath_k(\sigma_1,\sigma_2)&=\jmath_V^*\sigma_1-\jmath^*_W\sigma_2=\sigma_1|_{V\cap W}-\sigma_2|_{V\cap W}\\
&=f_W\omega+f_V\omega=(f_V+f_W)\omega=\omega.
\end{align*}
\end{proof}

A straightforward application of Corollary \ref{5ed} yields the following.

\begin{lemma}\label{4,5dr}
Let $V, W\subseteq M$ be abstract-definable open cover of $M$. With the same notation as in Lemma \ref{4dr}, for each $k\geq 0$, the diagram
\[
\xymatrix{
0\ar[r] &\Omega^{k+1}(M)\ar[r]^-{\imath_{k+1}} &\Omega^{k+1}(V)\oplus\Omega^{k+1}(W)\ar[r]^-{\jmath_{k+1}} &\Omega^{k+1}(V\cap W)\ar[r] &0\\
0\ar[r] &\Omega^{k}(M)\ar[r]^-{\imath_k}\ar[u]^{d_k} &\Omega^k(V)\oplus\Omega^k(W)\ar[r]^-{\jmath_k}\ar[u]^{D_k}& \Omega^k(V\cap W)\ar[r]\ar[u]^{d_k}  &0  
}
\]
is commutative, where $D_k\colon \Omega^k(V)\oplus \Omega^k(W)\to \Omega^{k+1}(V)\oplus \Omega^{k+1}(W)$ is given by $D_k(\omega,\tau)\mathrel{\mathop:}= (d_k\omega, d_k\tau)$ and $\imath_k$, $\jmath_k$ are as those defined in Lemma \ref{4dr}.
\end{lemma}

Lemmas \ref{4dr} and \ref{4,5dr} imply that 
$$
0\to \Omega^*(M)\to \Omega^*(V)\oplus \Omega^*(W)\to \Omega^*(V\cap W)\to 0
$$
is a short exact sequence of cochain complexes. Hence, by the Zig-zag lemma (Theorem 25.6, \cite{Tu2010}, p. 285) and the fact that 
$$
H^k(V)\oplus H^k(W)\cong \frac{\ker(D_k\colon \Omega^k(V)\oplus \Omega^k(W)\to \Omega^{k+1}(V)\oplus \Omega^{k+1}(W))}{\text{im}(D_{k-1}\colon \Omega^{k-1}(V)\oplus \Omega^{k-1}(W)\to \Omega^k(V)\oplus \Omega^k(W))}
$$ 
we obtain the Mayer-Vietoris sequence for o-minimal de Rham cohomology.

\begin{theorem}\label{5dr}
Let $V,W\subseteq M$ be abstract-definable open sets covering $M$. With the same notation as in Lemmas \ref{4dr} and \ref{4,5dr}, there exists a long exact sequence, the \emph{Mayer-Vietoris sequence},
$$
\cdots\to H^k(M)\stackrel{\imath_k^\sharp}{\to}H^k(V)\oplus H^k(W)\stackrel{\jmath_k^\sharp}{\to} H^k(V\cap W)\stackrel{d_k^\sharp}{\to}H^{k+1}(M)\to\cdots,
$$
where 
$$
\imath_k^\sharp([\omega])\mathrel{\mathop:}= ([\imath_V^*\omega],[\imath_W^*\omega]),
$$
$$\jmath_k^\sharp ([\omega],[\tau])\mathrel{\mathop:}= [\jmath^*_W\tau-\jmath^*_V\omega],
$$ 
and 
$$
d_k^\sharp([\omega])\mathrel{\mathop:}=[\imath_{k+1}^{-1}(D_k(\jmath_k^{-1}(\omega)))].
$$
Here, $\imath_{k+1}^{-1}(D_k(\jmath_k^{-1}(\omega)))$ means a chosen element $\tau$ in $\ker (d_{k+1})$ such that $\imath_{k+1}(\tau)$ is the value of $D_k$ at an element in $\jmath_{k}^{-1}(\omega)$. 
\end{theorem}

From the fact that $\Omega^k(M)=0$ for $k<0$, it follows that the Mayer-Vietoris sequence starts 
$$
0\to H^0(M)\to H^0(V)\oplus H^0(W)\to \cdots
$$



\begin{definition}\label{6dr}
Two abstract-definable $\mathcal{C}^\infty$ maps $F,G\colon M\to N$ are said to be \textit{abstract-definably homotopic}, and denoted by $F\simeq G$, if there is an abstract-definable $\mathcal{C}^\infty$ map $H\colon M\times \mathbb{R}\to N$ such that 
$$
H(x,0)=F(x),\ \text{and}\ H(x,1)=G(x)\ \text{for all}\ x\in M.
$$
The map $H$ is called an \textit{abstract-definable $\mathcal{C}^\infty$ homotopy} from $F$ to $G$. 
\end{definition}

\begin{definition}\label{7dr}
An abstract-definable $\mathcal{C}^\infty$ map $F\colon M\to N$ is said to be an \textit{abstract-definable $\mathcal{C}^\infty$ homotopy equivalence} if there exists an abstract-definable $\mathcal{C}^\infty$ map $G\colon N\to M$ such that $G\circ F\simeq \text{id}_M$ and $F\circ G\simeq \text{id}_N$. In this case, we also say that $M$ is \textit{abstract-definably homotopy equivalent} to $N$. If $M$ is abstract-definably homotopy equivalent to a point, then $M$ is called \textit{abstract-definably contractible}.
\end{definition}

The invariance of the o-minimal de Rham cohomology under abstract-definable homotopy,  the \textit{Homotopy Axiom} for o-minimal de Rham cohomology, does not hold in general as we will see below. In order to help us verify such an assertion we state this o-minimal version of the Homotopy Axiom and derive an easy consequence.

\begin{theorem}[Homotopy Axiom]\label{8dr}
Let $f,g\colon N\to M$ be abstract-definable $\mathcal{C}^\infty$ maps. If $f\simeq g$, then the induced maps in cohomology $f^\sharp, g^\sharp\colon H^*(M)\to H^*(N)$ agree with each other.
\end{theorem}

One immediate consequence of Theorem \ref{8dr} is the fact that, for each $k\geq 0$, $H^k(M)=H^k(N)$ whenever $M$ is abstract-definably homotopy equivalent to $N$. Thus, if $M$ is abstract-definably contractible, there exists a point $a\in M$ such that $H^k(M)=H^k(a)$ for $k\geq 0$. Hence $H^k(M)=0$ if $k>0$ (Theorem \ref{3dr}) and $H^0(M)=\mathbb{R}$ (Theorem \ref{2dr}). This proves the following corollary, also known as \textit{Poincar\'e's lemma}.

\begin{corollary}\label{9dr}
If $M$ is definably contractible, then $H^0(M)=\mathbb{R}$ and $H^k(M)=0$, for each $k\geq 1$.
\end{corollary}

Corollary \ref{9dr} picks o-minimal expansions of the real field (which admit smooth cell decomposition and define the exponential function) as candidates for which Theorem \ref{8dr} holds. These are to some extent large o-minimal structures, \textit{i.e.}, those that define sufficiently many primitives. In this sense, the exponential real field $\mathbb{R}_{\text{exp}}$ is not a large o-minimal structure, since $\mathbb{R}$ is definably contractible and $H^1(\mathbb{R})=1$ (see the remark right after Theorem \ref{3dr}).

In the sequel, we introduce an enlarged o-minimal structure, in the sense of what we just discoursed, and fix some assumptions in order to give a proof of Theorem \ref{8dr} in that setting.

Suppose from now on $\mathcal{R}$ is an o-minimal expansion of the real field $\mathbb{R}$, and let $U$ be a definable (in $\mathcal{R}$) open subset of $\mathbb{R}^n$. Recall from \cite{jospei2012} (p. 2) that a $\mathcal{C}^1$ function $f\colon U\to \mathbb{R}$ is \textit{Pfaffian over $\mathcal{R}$} if there exist definable (in $\mathcal{R}$) $\mathcal{C}^1$ functions $P_i\colon U\times \mathbb{R}\to\mathbb{R}$ for $i=1,\ldots,n$ such that 
$$
\frac{\partial f}{\partial r^i}(x)=P_i(x,f(x)),\ x\in U.
$$
Denote by $\mathcal{L}(\mathcal{R})$ the collection of all total functions $f\colon \mathbb{R}^n\to \mathbb{R}$ for all $n\in \mathbb{N}$ that are Pfaffian over $\mathcal{R}$. Set $\mathcal{R}_0\mathrel{\mathop :}= \mathcal{R}$, and for each $i\geq 0$ let $\mathcal{R}_{i+1}$ be the expansion of $\mathcal{R}_i$ by all functions in $\mathcal{L}(\mathcal{R}_i)$. Let $\mathcal{L}$ be the union of all $\mathcal{L}(\mathcal{R}_i)$ and let $\mathcal{P}(\mathcal{R})$ be the expansion of $\mathcal{R}$ by all the functions in $\mathcal{L}$. We call the structure $\mathcal{P}(\mathcal{R})$ the \textit{Pfaffian closure} of $\mathcal{R}$. Both Theorem 4.1 (\cite{spei1998}) and Theorem 1 (\cite{jospei2012}) imply that the Pfaffian closure $\mathcal{P}(\mathcal{R})$ of $\mathcal{R}$ is o-minimal. 
\\

\textit{From now on ``definable'' we mean ``definable in $\mathcal{P}(\mathcal{R})$ with parameters in $\mathbb{R}$'', where $\mathcal{P}(\mathcal{R})$ denotes the Pfaffian closure of $\mathcal{R}$, and by ``$\mathcal{R}$-definable'' we mean ``definable in $\mathcal{R}$ with parameters in $\mathbb{R}$''}.
\\

Recall from \cite{LionSpeissegger} that $\mathcal{P}(\mathcal{R})$ admits smooth cell decomposition.

The following assertion, known as \textit{Br\"ocker's question}, was pointed out to us by P. Speissegger.
\\

\noindent\textbf{Claim} (Br\"ocker's question)\label{Brockersproblem}
 For any continuous function $b\colon \mathbb{R}^n\times \mathbb{R}\to \mathbb{R}$ which is definable in an o-minimal expansion $\mathfrak{R}$ of $\mathbb{R}$, the function $B\colon \mathbb{R}^n\to \mathbb{R}$, given by 
 $$B(x)\mathrel{\mathop:}=\int_0^1b(x,t)dt,
 $$ 
 is definable  in an o-minimal expansion $\widetilde{\mathfrak{R}}$ of $\mathfrak{R}$.
 \\
 
 The above statement has been first proved for the case in which $\mathfrak{R}=\mathbb{R}_{\text{an}}$ and $\widetilde{\mathfrak{R}}=\mathbb{R}_{\text{an,exp}}$ by J.-M. Lion and J.-P. Rolin (\cite{LionRolin}). In \cite{kaiser2012} (Theorem 1.9), T. Kaiser formulated and proved a generalization of Br\"ocker's question. Namely, the Lebesgue measure $\lambda_n$ on $\mathbb{R}^n$ satisfies in particular the following condition: there exists an o-minimal expansion $\widetilde{\mathfrak{R}}$ of $\mathbb{R}_{\text{an}}^{\mathbb{R}_\text{alg}}$ such that for any definable (in $\mathbb{R}_{\text{an}}^{\mathbb{R}_\text{alg}}$) function $f\colon \mathbb{R}^m\times \mathbb{R}^n\to \mathbb{R}$ the set $\infty(f,\lambda_n)\mathrel{\mathop:}=\{x\in \mathbb{R}^m\,:\, \int_{\mathbb{R}^n}|f(x,t)|d\lambda_n(t)=\infty\}$ is definable in $\mathbb{R}_{\text{an}}^{\mathbb{R}_\text{alg}}$, and the function $x\mapsto \int_{\mathbb{R}^n}f(x,t)d\lambda_n(t)\colon \mathbb{R}^m\setminus \infty(f,\lambda_n)\to \mathbb{R}$ is definable in $\widetilde{\mathfrak{R}}$. (Recall that $\mathbb{R}_{\text{an}}$ indicates the o-minimal expansion of the real field $\mathbb{R}$ by all restricted real analytic functions, and $\mathbb{R}_{\text{an}}^{\mathbb{R}_\text{alg}}$ denotes the expansion of $\mathbb{R}_{\text{an}}$ by all power functions with exponent in $\mathbb{R}_{\text{alg}}$, the field of real algebraic numbers.)   

\begin{lemma}\label{11,1odrc}
Assume that the Br\"ocker's question holds for any o-minimal expansion $\mathfrak{R}$ of $\mathbb{R}_{\exp}$ and for $\widetilde{\mathfrak{R}}$ taken to be the Pfaffian closure of $\mathfrak{R}$. If $b\colon \mathbb{R}^n\times \mathbb{R}\to \mathbb{R}$ is a definable $\mathcal{C}^\infty$ function, so is $B\colon \mathbb{R}^n\to \mathbb{R}$, where $B(x)\mathrel{\mathop:}= \int_0^1b(x,t)dt$. As a consequence, given an abstract-definable $\mathcal{C}^\infty$ function $b\colon N\times \mathbb{R}\to \mathbb{R}$, the function $\overline{b}\colon N\to \mathbb{R}$ given by $y\mapsto \int_0^1 b(y,t)dt$ is abstract-definable $\mathcal{C}^\infty$ as well.
\end{lemma}
\begin{proof}
Since $b$ is definable in $\mathcal{P}(\mathcal{R})$, there is some $i\geq 0$ such that $b$ is definable in $\mathcal{R}_i$, where $\mathcal{R}_0\mathrel{\mathop:}= \mathcal{R}$ (see the definition of Pfaffian closure above). Note, from the comments following Theorem 1 (\cite{jospei2012}, p. 2), that the Pfaffian closure of $\mathcal{R}$ can be obtained by adding only definable $\mathcal{C}^\infty$ total functions. In particular, $\mathcal{R}_i$ is an o-minimal expansion of $\mathbb{R}_{\exp}$ which admits smooth cell decomposition. Hence, the assumption implies that $B$ is definable in $\mathcal{P}(\mathcal{R}_i)$. The conclusion that $B$ is definable in $\mathcal{P}(\mathcal{R})$ follows from the fact that $\mathcal{P}(\mathcal{R}_i)$ and $\mathcal{P}(\mathcal{R})$ are interdefinable. The smoothness of $B$ is ensured, for instance, by Theorem C.14 (\cite{Lee}, p. 648).

Now, observe that for any chart $(V,\psi)$ on $N$ the composition $\overline{b}\circ \psi^{-1}\colon \mathbb{R}^n\to \mathbb{R}$ agrees with 
$$
\psi(z)\mapsto \int_0^1 (b\circ (\psi\times \text{id}_{\mathbb{R}})^{-1})(\psi(z),t)dt.
$$ 
(Here we assumed the codomains of the charts $\psi\colon V\to \psi(V)$ are the whole $\mathbb{R}^n$, see Remark \ref{remarksection1}.) By hypothesis, $b\circ (\psi\times \text{id}_{\mathbb{R}})^{-1}$ is a definable $\mathcal{C}^\infty$ function on $\mathbb{R}^n\times \mathbb{R}$, and from the first part of the lemma it follows that $\overline{b}\circ \psi^{-1}$ is also definable $\mathcal{C}^\infty$. This proves that $\overline{b}$ is an abstract-definable $\mathcal{C}^\infty$ function.
\end{proof}

\begin{lemma}\label{lemmaforhomotopyaxiom}
Let $U$ be an abstract-definable open subset of $N$, and let $\omega\in \Omega^k(U)$ with
$$
\text{supp}(\omega)\subseteq F\subseteq U,
$$
where $F$ is an abstract-definable closed subset of $N$. Then, $\omega$ can be extended to an abstract-definable $\mathcal{C}^\infty$ $k$-form $\widetilde{\omega}$ on $N$. 
\end{lemma}
\begin{proof}
Let $\omega$ be an abstract-definable $\mathcal{C}^\infty$ $k$-form on $U$, and suppose $F$ is an abstract-definable closed set with $\text{supp}(\omega)\subseteq F\subseteq U$. By the abstract-definable smooth version of Proposition \ref{7pu}, there exists an abstract-definable $\mathcal{C}^\infty$ function $\rho\colon M\to \mathbb{R}$ such that $0\leq \rho\leq 1$, $\rho|_F=1$, and $\text{supp}(\rho)\subseteq U$. Take $\widetilde{\omega}\colon N\to \bigwedge^kT^*N$ to be 
\[
\widetilde{\omega}(x)\mathrel{\mathop:}=
\begin{cases}
\rho(x)\omega(x),& \text{if}\ x\in U\\
0,&  \text{otherwise}
\end{cases}.
\]
Firstly, note that $\widetilde{\omega}$ is well-defined as an abstract-definable $\mathcal{C}^\infty$ map from $N$ to $\bigwedge^kT^*N$. Also, for any $x\in U$, if $x\in \text{supp}(\omega)\subseteq F$ then $\widetilde{\omega}(x)=\rho(x)\omega(x)=\omega(x)$; and if $x\in U\setminus \text{supp}(\omega)$ then $\omega(x)=0=\rho(x)\cdot 0 = \rho(x)\cdot \omega(x)=\widetilde{\omega}(x)$. In other words, $\widetilde{\omega}|_U=\omega$.
\end{proof}

\textit{For the remainder of the section, we assume that the Br\"ocker's question holds for any o-minimal expansion $\mathfrak{R}$ of $\mathbb{R}_{\exp}$ and for $\widetilde{\mathfrak{R}}$ taken to be the Pfaffian closure of $\mathfrak{R}$}.

\begin{proof}[Proof of Theorem \ref{8dr}]
Let $f,g\colon N\to M$ be abstract-definable $\mathcal{C}^\infty$ maps such that $f\simeq g$. Then, there exists an abstract-definable $\mathcal{C}^\infty$ map $H\colon N\times \mathbb{R}\to M$ satisfying $H(\iota_0(x))=f(x)$ and $H(\iota_1(x))=g(x)$, for all $x\in N$, where $\iota_t\colon N\to N\times \mathbb{R}$ is the abstract-definable $\mathcal{C}^\infty$ map given by $x\mapsto (x,t)$, for each fixed $t\in \mathbb{R}$. Since $^\sharp$ is a contravariant functor, $\iota_0^\sharp\circ H^\sharp=f^\sharp$ and $\iota_1^\sharp\circ H^\sharp=g^\sharp$. Hence, in order to prove the theorem it suffices to show that $\iota_0^\sharp$ and $\iota_1^\sharp$ are the same. Note that if there is a cochain homotopy $K\mathrel{\mathop:}= \{K_k\colon \Omega^k(N\times \mathbb{R})\to \Omega^{k-1}(N)\}_k$ between the induced pullback maps $\imath_1^*\mathrel{\mathop:}= \{(\imath_1^*)_k\colon \Omega^k(N\times \mathbb{R})\to \Omega^k(N)\}_k$ and $\imath_0^*\mathrel{\mathop:}= \{(\imath_0^*)_k\colon \Omega^k(N\times \mathbb{R})\to \Omega^k(N)\}_k$, then the induced maps in cohomology $\imath^\sharp_0$ and $\imath_1^\sharp$ agree with each other. In the remainder of the proof, we thus focus on establishing linear maps $K_k\colon \Omega^k(N\times \mathbb{R})\to \Omega^{k-1}(N)$ for $k\geq 0$, which satisfy the equality
\begin{equation}\label{1eo-mincoho}
d_{k-1}\circ K_k+K_{k+1}\circ d_k=(\imath_1^*)_k-(\imath_0^*)_k,
\end{equation}
where $d_k$ denotes the exterior derivative on $\Omega^k(N\times \mathbb{R})$ and, as an abuse of notation, on $\Omega^k(N)$ as well.
\\

\noindent\textbf{Claim 1}\label{claimhomotopyaxiom}
Every abstract-definable $\mathcal{C}^\infty$ $k$-form on $N\times \mathbb{R}$ can be written as a finite sum of abstract-definable $\mathcal{C}^\infty$ forms of the types: 
\begin{enumerate}
\item[(I)] $a\,\pi^*\eta$;
\item[(II)] $b\, dt\wedge \pi^*\tau$,
\end{enumerate}
 where $a,b$ are abstract-definable $\mathcal{C}^\infty$ functions on $N\times \mathbb{R}$, $\pi\colon N\times \mathbb{R}\to N$ is the projection onto the first factor, $\eta$ is an abstract-definable $\mathcal{C}^\infty$ $k$-form on $N$, and $\tau$ is an abstract-definable $\mathcal{C}^\infty$ $(k-1)$-form on $N$. 

\begin{proof}[Proof of Claim 1]
Let $\mathcal{B}$ denote the collection $\{\psi_j\colon V_j\to \psi_j(V_j)\subseteq \mathbb{R}^n\}_{j\in J}$, and fix an abstract-definable $\mathcal{C}^\infty$ $k$-form $\omega$ on $N\times \mathbb{R}$. Let $\{\rho_j\}_{j\in J}$ be an abstract-definable $\mathcal{C}^\infty$ partition of unity subordinate to $\mathcal{B}$ (see Lemma 4.6, \cite{fischer2008}). By the abstract-definable smooth version of Proposition \ref{7pu}, there exists a finite family $\{g_j\}_{j\in J}$ of abstract-definable $\mathcal{C}^\infty$ functions on $N$ such that for each $j$: $0\leq g_j\leq 1$, $g_j|_{\text{supp}(\rho_j)}=1$, and $\text{supp}(g_j)\subseteq V_j$. Note that $\{\pi^{-1}(V_j)=V_j\times \mathbb{R}\}_{j\in J}$ is an abstract-definable open cover of $N\times \mathbb{R}$. Also, $\{\pi^*\rho_j\}_{j\in J}$ is an absctract-definable $\mathcal{C}^\infty$ partition of unity subordinate to $\{\pi^{-1}(V_j)\}_{j\in J}$ in the following sense:
\begin{enumerate}
\item[(1)] each $\pi^*\rho_j\colon N\times \mathbb{R}\to \mathbb{R}$ is an abstract-definable $\mathcal{C}^\infty$ nonnegative function;
\item[(2)] $\text{supp}(\pi^*\rho_j)\subseteq \pi^{-1}(V_j)$, for each $j\in J$;
\item[(3)] $\sum_{j\in J}\pi^*\rho_j=1$.
\end{enumerate}
Indeed, (1) follows immediately from the definition of pullback of abstract-definable $\mathcal{C}^\infty$ $0$-forms, that is, $\pi^*\rho_j=\rho_j\circ \pi$. Now, observe that 
$$
\pi(\{(x,r)\in N\times \mathbb{R}\,:\, \rho_j(x)\neq 0\})=\{x\in N\,:\, \rho_j(x)\neq 0\}\subseteq \text{supp}(\rho_j)\subseteq V_j.
$$
Consequently, 
$$
\{(x,r)\in N\times \mathbb{R}\,:\, \rho_j(x)\neq 0\}\subseteq \pi^{-1}(\text{supp}(\rho_j))\subseteq \pi^{-1}(V_j).
$$
Since $\pi^{-1}(\text{supp}(\rho_j))$ is closed in $N\times \mathbb{R}$, 
$$
\text{supp}(\pi^*\rho_j)=\text{cl}_{N\times \mathbb{R}}(\{(x,r)\in N\times \mathbb{R}\,:\, \rho_j(x)\neq 0\})\subseteq \pi^{-1}(\text{supp}(\rho_j))\subseteq \pi^{-1}(V_j).
$$
Thus, (2) follows. Finally, because $\sum_{j\in J} \pi^*\rho_j(z,t)=\sum_{j\in J} \rho_j(\pi(z,t))=\sum_{j\in J}\rho_j(z)$ for all $(z,t)\in N\times \mathbb{R}$, the validity of (3) is thereby obtained. 

By virtue of (1)-(3), we can write $\omega$ as
$$
\omega=\left(\sum\limits_{j\in J} \pi^*\rho_j\right)\omega=\sum\limits_{j\in J}(\pi^*\rho_j)\omega=\sum\limits_{j\in J}\omega_j,
$$
where $\omega_j\mathrel{\mathop:}= (\pi^*\rho_j)\omega\in \Omega^k(N\times \mathbb{R})$. Note that 
\begin{equation}\label{2eo-mincoho}
\text{supp}(\omega_j)\subseteq \text{supp}(\pi^*\rho_j)\cap \text{supp}(\omega)\subseteq \text{supp}(\pi^*\rho_j)\subseteq \pi^{-1}(V_j).
\end{equation}

If we show that each $\omega_j$ can be written as a finite sum of type-(I) and type-(II) abstract-definable $\mathcal{C}^\infty$ forms, then we are done.

Let $(V_j,\psi_j)=(V_j,y^1,\ldots,y^n)$ be a chart in $\mathcal{B}$. Since $\pi^{-1}(V_j)=V_j\times \mathbb{R}$, the collection $\{(\pi^{-1}(V_j), \pi^*y^1,\ldots,\pi^*y^n,t)\}_{j\in J}$ forms an abstract-definable $\mathcal{C}^\infty$ atlas on $N\times \mathbb{R}$, where $t$ is the projection $(x,r)\mapsto r\colon N\times \mathbb{R}\to \mathbb{R}$ restricted to $\pi^{-1}(V_j)$. Thus, on $\pi^{-1}(V_j)$, the abstract-definable $\mathcal{C}^\infty$ $k$-form $\omega_j$ can be written uniquely as
\begin{align*}
\omega_j&=\sum\limits_{1\leq i_1<\cdots<i_k\leq n} a_{i_1\cdots i_k}d(\pi^*y^{i_1})\wedge \cdots \wedge d(\pi^*y^{i_k})\\
&+ \sum_{1\leq l_1<\cdots<l_{k-1}\leq n} b_{l_1\cdots l_{k-1}} dt\wedge d(\pi^*y^{l_1})\wedge \cdots \wedge d(\pi^*y^{l_{k-1}})\\
&=\sum\limits_{I} a_I\pi^*dy^I + \sum_{L} b_L dt\wedge \pi^*dy^L,
\end{align*}
after a rearrangement of the terms, where $I$ and $L$ denotes respectively $i_1<\ldots<i_k$ and $l_1<\ldots<l_{k-1}$, $dy^I\mathrel{\mathop:}= dy^{i_1}\wedge \cdots \wedge dy^{i_k}$, $dy^L\mathrel{\mathop:}= dy^{l_1}\wedge \cdots \wedge dy^{l_{k-1}}$, and $a_I$, $b_L$ are abstract-definable $\mathcal{C}^\infty$ functions on $\pi^{-1}(V_j)$. Once by (\ref{2eo-mincoho}) we have
\begin{align*}
\text{supp}(a_{i_1\cdots i_k}), \text{supp}(b_{l_1\cdots l_{k-1}}) &\subseteq \text{cl}_{\pi^{-1}(V_j)}(\{(x,t)\in \pi^{-1}(V_j)\,:\, \omega_j(x,t)\neq 0\})\\
&\subseteq \text{supp}(\omega_j)\subseteq \text{supp}(\pi^*\rho_j)\subseteq \pi^{-1}(V_j),
\end{align*}
we may then use Lemma \ref{lemmaforhomotopyaxiom} (with $F$ taken to be $\text{supp}(\pi^*\rho_j)$) to obtain abstract-definable $\mathcal{C}^\infty$ $0$-forms $\widetilde{a}_I$, $\widetilde{b}_L\colon N\times \mathbb{R}\to \mathbb{R}$ which extend $a_I$ and $b_L$ by zero, respectively. Note that we cannot proceed similarly for the abstract-definable $\mathcal{C}^\infty$ forms $dy^I$, $dy^L$ by applying Lemma \ref{lemmaforhomotopyaxiom}, since the (topological) closure of the subsets of their domains in which $dy^I$ and $dy^L$ do not vanish coincide with their domains $V_j$, and this is not a closed subset of $N$. Nevertheless, we may get around this problem through the multiplication of $\omega_j$ by $\pi^*g_j$. In fact, because $\pi^*g_j=1$ on $\text{supp}(\pi^*\rho_j)$ and $\text{supp}(\omega_j)\subseteq \text{supp}(\pi^*\rho_j)$, the equality $\omega_j=(\pi^*g_j)\omega_j$ holds. Therefore, on $\pi^{-1}(V_j)$, $\omega_j$ can be rewritten as 
\begin{align}\label{3eo-mincoho}
\omega_j=(\pi^*g_j)\omega_j&=\sum\limits_Ia_I(\pi^*g_j) \pi^*dy^I+\sum\limits_Lb_Ldt\wedge (\pi^*g_j)\pi^*dy^L\\
&=\sum\limits_Ia_I\pi^*(g_jdy^I)+\sum\limits_Lb_Ldt\wedge \pi^*(g_jdy^L)\nonumber
\end{align}
Once $\text{supp}(g_j|_{V_j})\subseteq \text{supp}(g_j)\subseteq V_j$, we obtain by Lemma \ref{lemmaforhomotopyaxiom} extensions $\eta_j$ and $\tau_j$ by zero of $g_jdy^I$ and $g_jdy^L$ to $N$, respectively. 

Finally, observe that the support of $\omega_j\in \Omega^k(N\times \mathbb{R})$ is contained in $\pi^{-1}(V_j)$ as well as the supports of each abstract-definable $\mathcal{C}^\infty$ form among $a_I$, $b_L$, $\pi^*(g_jdy^I)$, and $dt\wedge\pi^*(g_jdy^L)$. Thus, $\omega_j$ equals the extension by zero of $\omega_j|_{\pi^{-1}(V_j)}$ to $N\times \mathbb{R}$, which in turn equals the sum of the products of the extension by zero (to $N\times \mathbb{R}$) of each term in (\ref{3eo-mincoho}), in other words,
\begin{equation}\label{4eo-mincoho}
\omega_j=\sum\limits_I \widetilde{a}_I \pi^*\eta_j+\sum\limits_L \widetilde{b}_L dt\wedge \pi^*\tau_j,
\end{equation} 
with $\widetilde{a}_I,\widetilde{b}_L\in \Omega^0(N\times \mathbb{R})$, $\eta_j\in \Omega^k(N)$, and $\tau_j\in \Omega^{k-1}(N)$.
\end{proof}

 Define $K_k\colon \Omega^k(N\times \mathbb{R})\to \Omega^{k-1}(N)$ by: 
 \begin{enumerate}
 \item[(i)] $K_k(a\, \pi^*\eta)\mathrel{\mathop:}= 0$, on type-(I) abstract-definable $\mathcal{C}^\infty$ $k$-forms;
 \item[(ii)] $ K_k(b\, dt\wedge \pi^*\tau)\mathrel{\mathop:}= (\int_0^1b(y,t)dt)\tau$, on type-(II) abstract-definable $\mathcal{C}^\infty$ $k$-forms;
 \item[(iii)] $K_k$ is extended linearly.
 \end{enumerate}
 
  After fixing an abstract-definable $\mathcal{C}^\infty$ partition of unity $\{\rho_j\}_j$ subordinate to $\mathcal{B}$, and a finite collection $\{g_j\}_{j\in J}$ of abstract-definable $\mathcal{C}^\infty$ functions on $N$, we can express $\omega\in \Omega^{k}(N\times \mathbb{R})$ as a sum $\omega=\sum_j\omega_j$, where $\omega_j$ is decomposed uniquely into 
  $$
  \sum\limits_{j,I} a^j_I\pi^*\eta_j+\sum\limits_{j,L}b^j_Ldt\wedge \pi^*\tau_j
  $$
 like in (\ref{4eo-mincoho}) (see the proof of Claim 1). So,  $K_k(\omega)=\sum_{j,L}(\int_0^1b^j_L(x,t)dt)\tau_j$. (Lemma \ref{11,1odrc} shows that each $\int_0^1b^j_L(x,t)dt$ is an abstract-definable $\mathcal{C}^\infty$ function on $N$, therefore $K_k(\omega)$ lies indeed in $\Omega^{k-1}(N)$.)
 
Let us now check (\ref{1eo-mincoho}). Fix a chart $(V\times \mathbb{R},\pi^*y^1,\ldots,\pi^*y^n,t)$ on $N\times \mathbb{R}$. For type-(I) abstract-definable $\mathcal{C}^\infty$ $k$-forms, we have
\begin{align*}
&(K_{k+1}\circ d_k)(a \pi^*\eta)=\\
&K_{k+1}(d_0 a \wedge \pi^*\eta+a d_k(\pi^*\eta))=\\
&K_{k+1}\left(\left(\sum\limits_{i=1}^n\frac{\partial a}{\partial \pi^*y^i}d(\pi^*y^i)+\frac{\partial a}{\partial t} dt\right)\wedge \pi^*\eta+a\pi^*(d_k\eta)\right)=\\
&K_{k+1}\left(\frac{\partial a}{\partial t}dt\wedge \pi^*\eta+\sum\limits_{i=1}^n \frac{\partial a}{\partial \pi^*y^i}\pi^*(d_0y^i)\wedge \pi^*\eta+a \pi^*d_k\eta\right)=\\
&K_{k+1}\left(\frac{\partial a}{\partial t}dt\wedge \pi^*\eta\right)+\sum\limits_{i=1}^nK_{k+1}\left(\frac{\partial a}{\partial \pi^*y^i}\pi^*(d_0y^i\wedge \eta)\right)+K_{k+1}(a\pi^*d_k\eta)=\\
&\left(\int_0^1\frac{\partial a}{\partial t}dt\right)\eta=
(a(x,1)-a(x,0))\eta=\imath^*_1a(\imath^*_1\pi^*\eta)-\imath^*_0a(\imath_0^*\pi^*\eta)=(\imath_1^*-\imath_0^*)(a\pi^*\eta)
\end{align*}
 and  $(d_{k-1}\circ K_k)(a\pi^*\eta)=d_{k-1}(0)=0$. Hence, 
 $$
 (d_{k-1}\circ K_k+K_{k+1}\circ d_k)(a\pi^*\eta)=(\imath_1^*-\imath_0^*)(a\pi^*\eta).
 $$ 
 For type-(II) abstract-definable $\mathcal{C}^\infty$ $k$-forms, we get
 \begin{align*}
 &(d_{k-1}\circ K_k)(b\,dt\wedge \pi^*\tau)=\\
 &d_{k-1}\left(\left(\int_0^1b(y,t)dt\right)\tau\right)=\\
 &d_{k-1}\left(\int_0^1b(y,t)dt\right)\wedge \tau+\left(\int_0^1b(y,t)dt\right)d_{k-1}\tau=\\
 &\sum\limits_{i=1}^n \frac{\partial}{\partial y^i}\left(\int_0^1 b(y,t)dt\right)dy^i\wedge \tau+\left(\int_0^1b(y,t)dt\right)d_{k-1}\tau
 \end{align*}
and also
\begin{align*}
&(K_{k+1}\circ d_k)(b\,dt\wedge \pi^*\tau)=\\
&K_{k+1}(d_1(b\,dt)\wedge \pi^*\tau-b\,dt\wedge d_k(\pi^*\tau))=\\
&K_{k+1}(d_1b\wedge dt\wedge \pi^*\tau)-K_{k+1}(b\,dt\wedge \pi^*d_{k-1}\tau)=\\
&K_{k+1}\left(\sum\limits_{i=1}^n\frac{\partial b}{\partial \pi^*y^i}d\pi^*y^i\wedge dt\wedge \pi^*\tau\right)-K_{k+1}(b\,dt\wedge \pi^*d_{k-1}\tau)=\\
&-\sum\limits_{i=1}^nK_{k+1}\left(\frac{\partial b}{\partial \pi^*y^i}dt\wedge \pi^*(dy^i\wedge \tau)\right)-\left(\int_0^1b(y,t)dt\right)d_{k-1}\tau=\\
&-\sum\limits_{i=1}^n\left(\int_0^1\frac{\partial b}{\partial \pi^*y^i}(y,t)dt\right)dy^i\wedge \tau-\left(\int_0^1b(y,t)dt\right)d_{k-1}\tau=\\
&-\sum\limits_{i=1}^n\frac{\partial }{\partial y^i}\left(\int_0^1b(y,t)dt\right)dy^i\wedge \tau-\left(\int_0^1b(y,t)dt\right)d_{k-1}\tau,
\end{align*}
where the last equality followed from the differentiation under the integral sign. Furthermore, $\imath_1^*(b\, dt\wedge \pi^*\tau)=b(y,1)\,\imath_1^*(dt)\wedge \imath_1^*(\pi^*\tau)=0$, since $\imath_1^*(dt)=d(\imath_1^*t)=d(1)=0$. Similarly, $\imath_0^*(b\, dt\wedge \pi^*\tau)=0$. Therefore, 
$$
(d_{k-1}\circ K_k+K_{k+1}\circ d_k)(b\, dt\wedge \pi^*\tau)=0=(\imath_1^*-\imath_0^*)(b\, dt\wedge \pi^*\tau).
$$
This finishes the proof.
\end{proof}

\section{Final remarks and future works}

The validity of the Homotopy Axiom (Theorem \ref{8dr}) is uniquely conditioned to the abstract-definability of the function $y\mapsto \int_0^1 b(y,t)dt$, arisen when defining the cochain homotopy $(K_k\colon \Omega^k(N\times \mathbb{R})\to \Omega^{k-1}(N))_k$ (see (i)-(iii) below Claim \ref{claimhomotopyaxiom} in the proof of Theorem \ref{8dr} for the Pfaffian closure $\mathcal{P}(\mathcal{R})$). In turn, as pointed out by Lemma \ref{11,1odrc}, such an abstract-definability question reduces to whether Br\"ocker's question holds. Recall from Definition 1.3 (\cite{kaiser2012}) that, given an o-minimal expansion $\mathfrak{R}$ of $\mathbb{R}$ and a Borel measure $\mu$ on $\mathbb{R}^n$, an \textit{$\mathfrak{R}$-integrating o-minimal structure of $\mu$} is an o-minimal expansion $\widetilde{\mathfrak{R}}$ of $\mathfrak{R}$ such that the set $\infty(f,\mu)\mathrel{\mathop:}=\{x\in \mathbb{R}^m\,:\, \int_{\mathbb{R}^n}|f(x,t)|d\mu(t)=\infty\}$ and the function $x\mapsto \int_{\mathbb{R}^n}f(x,t)d\mu(t)\colon \mathbb{R}^m\setminus \infty(f,\mu)\to \mathbb{R}$ are both definable in $\widetilde{\mathfrak{R}}$, for every definable $f\colon \mathbb{R}^m\times \mathbb{R}^n\to \mathbb{R}$. Therefore, the Homotopy Axiom is true for each pair $(\mathfrak{R},\widetilde{\mathfrak{R}})$ of o-minimal expansions of the real field, with $\widetilde{\mathfrak{R}}$ an $\mathfrak{R}$-integrating o-minimal structure of the Borel measure on $\mathbb{R}$ given by the restriction on the Borel $\sigma$-algebra $\mathcal{B}(\mathbb{R})$ of Lebesgue measure $\lambda$. 
 
In \cite{rrr} we attempt to settle a result on the smoothing abstract-definable $\mathcal{C}^p$ manifolds, with $1\leq p<\infty$. Namely, any abstract definable $\mathcal{C}\sp{p}$ manifold has a compatible $\mathcal{C}\sp{p+1}$ atlas. This allows us to establish an o-minimal de Rham cohomology for the category of abstract-definable $\mathcal{C}^p$ manifolds, where $p$ is a positive integer, so we could remove the assumption on the fixed o-minimal structure $\mathcal{R}$ of admitting smooth cell decomposition.

A further step might be the formulation of a $\mathcal{C}^p$ singular cohomology for abstract-definable $\mathcal{C}^p$ manifolds where $1\leq p\leq \infty$, restricting Edmundo's work (\cite{edmundo2001}) on singular cohomology for the category of abstract-definable $\mathcal{C}^0$ manifolds and maps, with the ultimate goal of the establishment of a de Rham's theorem for the category of abstract-definable $\mathcal{C}^p$ manifolds and maps.


\begin{thebibliography}{88}

\bibitem{beot2001}
 Alessandro Berarducci and Margarita Otero, 
 Intersection theory for o-minimal manifolds, 
 \textit{Ann. Pure Appl. Logic} 107 (2001), no. 1--3, 87-119. MR1807841 (2001m:03074) 

\bibitem{bianconi1997}
Ricardo Bianconi,
Nondefinability results for expansions of the field of real numbers by the
exponential function and by the restricted sine function. 
\textit{J. Symbolic Logic} 62 (1997), no. 4, 1173--1178. MR1617985 (99k:03034)

\bibitem{rrr}
Ricardo Bianconi, Rodrigo Figueiredo and Robson A. Figueiredo,
On Whitney embedding of o-minimal manifolds (2019). (\href{https://arxiv.org/abs/1904.05403}{arXiv:1904.05403})

\bibitem{denef-vandendries}
Jan Denef and Lou van den Dries, $p$-adic and real subanalytic sets, \textit{Ann. of Math}. (2) 128 (1988), 79--138. MR0951508 (89k:03034)

\bibitem{vdDries1984}
Lou van den Dries, 
Remarks on Tarski's problem concerning $(\mathbb{R},+,\cdot,\exp)$, in \textit{Logic colloquium' 82}, pp. 97--121, Lolli, Longo, and Marcja, editors, North Holland, 1984.

\bibitem{vandendries-tame}
Lou van den Dries,
\textit{Tame topology and o-minimal structures.} (English summary)
London Mathematical Society Lecture Note Series, 248. Cambridge University Press, Cambridge, 1998. x$+$180 pp. ISBN: 0-521-59838-9 MR1633348 (99j:03001)

\bibitem{vandendries-miller-geometric}
Lou van den Dries and Chris Miller,
Geometric categories and o-minimal structures.
\textit{Duke Math. J.} 84 (1996), no. 2, 497--540.
MR1404337 (97i:32008)

\bibitem{edmundo2009}
M\'ario J. Edmundo,
Covering definable manifolds by open definable subsets. (English summary) \textit{Logic Colloquium 2005}, 18--25,
Lect. Notes Log., 28, \textit{Assoc. Symbol. Logic, Urbana, IL, 2008}.  MR2395801 (2009c:03032)

\bibitem{edmundo2001}
M\'ario J. Edmundo, 
O-minimal cohomology and definably compact definable groups.
(\href{https://arxiv.org/abs/math/0012050v2}{arXiv:math/0012050v2})

\bibitem{edmundo-etc2006}
M\'ario J. Edmundo, Gareth O. Jones, Nicholas J. Peatfield,
Sheaf cohomology in o-minimal structures. (English summary)
\textit{J. Math. Log.} 6 (2006), no. 2, 163--179.  MR2317425 (2008e:03063)

\bibitem{edmundo-etc2016}
M\'ario J. Edmundo, Marcello Mamino, and Luca Prelli, 
On definably proper maps. 
\textit{Fund. Math.} 233 (2016), no. 1, 1--36. MR3460632

\bibitem{edmundo-etc2008}
M\'ario J. Edmundo and Nicholas J. Peatfield, 
o-Minimal \v Cech cohomology,
\textit{Q. J. Math.} 59 (2008), no. 2, 213-220. MR2428077 (2009d:18007)

\bibitem{edmundo-woerheide2008}
M\'ario J. Edmundo, Arthur Woerheide,
Comparison theorems for o-minimal singular (co)homology. \textit{Trans. Amer. Math. Soc.} 360 (2008), no. 9, 4889--4912.
MR2403708 (2009f:03048)

\bibitem{rf2017}
Rodrigo Figueiredo,
\textit{O-minimal de Rham cohomology}, PhD Thesis (2017). (Available at \href{http://www.teses.usp.br/teses/disponiveis/45/45131/tde-28042019-181150/}{http://www.teses.usp.br/teses/disponiveis/45/45131/tde-28042019-181150/})

\bibitem{fischer2008} 
Andreas Fischer, 
Smooth functions in o-minimal structures. \textit{Advances in Mathematics} 218 (2008), 496--514. MR2407944 (2009e:03070)

\bibitem{gao2016}
Ziyang Gao,
About the mixed Andr\'e-Oort conjecture: reduction to a lower bound for the pure case,
\textit{C. R. Math. Acad. Sci. Paris} 354 (2016), no. 7, 659--663. MR3508560



\bibitem{jospei2012}
Gareth O. Jones and Patrick Speissegger,
Generating the Pfaffian closure with total Pfaffian functions, 
\textit{J. Log. Anal.} 4 (2012), Paper 5, pp. 6. MR2889825

\bibitem{kaiser2012}
Tobias Kaiser,
First order tameness of measures, 
\textit{Ann. Pure Appl. Logic} 163 (2012) 1903--1927.

\bibitem{KPS}
Julia Knight, Anand Pillay, and Charles Steinhorn,
Definable sets in ordered structures. II. 
\textit{Trans. AMS} 295 (1986), 593--605.

\bibitem{Lee} John M. Lee, 
\textit{Introduction to Smooth Manifolds}, 2nd Ed., Springer, New York, 2013.

\bibitem{legalrolin2009}
Olivier Le Gal and Jean-Philippe Rolin,
An o-minimal structure which does not admit $\mathcal{C}^\infty$-cellular decomposition, \textit{Ann. Inst. Fourier (Grenoble)} 59 (2009), no. 2, 543--562. MR2521427 (2010f:03038)

\bibitem{LionSpeissegger}
Jean-Marie Lion and Patrick Speissegger,
Analytic stratification in the Pfaffian closure of an o-minimal structure. 
 \textit{Duke Math. J.} \textbf{103} (2000), no. 2, 215--231.
 
 \bibitem{LionRolin}
Jean-Marie Lion and Jean-Philippe Rolin. 
Int\'egration des fonctions sous-analytiques et volumes des sous-ensembles sous-analytiques. 
\textit{Ann. Inst. Fourier} \textbf{48} (1998), 755--767.

\bibitem{ibjorgen2001}
Ib Madsen and Jorgen Tornehave, 
\textit{From Calculus to Cohomology. de Rham cohomology and characteristic classes}. Cambridge University Press, Cambridge (1997). viii+286 pp. ISBN: 0-521-58059-5; 0-521-58956-8. MR1454127 (98g:57040)

\bibitem{peterzil-etc1999}
 Ya'acov Peterzil and Charles Steinhorn, 
 Definable compactness and definable subgroups of o-minimal groups, 
 \textit{J. London Math. Soc.} (2) 59 (1999), no. 3, 769--786. MR1709079 (2000i:03055) 

\bibitem{pila2011}
Jonathan Pila,
O-minimality and the Andr\'e-Oort conjecture for $\mathbb{C}^n$.
\textit{Ann. of Math.} (2) 173 (2011), no. 3, 1779--1840. MR2800724 (2012j:11129)

\bibitem{pilatsimerman2013}
Jonathan Pila and Jacob Tsimerman,
The Andr\'e-Oort conjecture for the moduli space of abelian surfaces, 
\textit{Compos. Math.} 149 (2013), no. 2, 204--216. MR3020307

\bibitem{pilawilkie2006}
Jonathan Pila and Alex Wilkie, 
The rational points of a definable set, 
\textit{Duke Math. J.} 133 (2006), no. 3, 591--616. MR2228464 (2007f:03048)

\bibitem{PSI}
Anand Pillay and Charles Steinhorn, 
Definable sets in ordered structures I. \textit{Trans. AMS} 295 (1986), 565--592.
\bibitem{PSIII}
Anand Pillay and Charles Steinhorn,
Definable sets in ordered structures III. \textit{Trans. AMS} 309 (1988), 469--476.

\bibitem{shiota1986}
Masahiro Shiota,
Abstract Nash manifolds, 
\textit{Proc. Amer. Math. Soc.} 96 (1986), no. 1, 155--162. MR813829 (87e:58003)

\bibitem{spei1998}
Patrick Speisseggger,
The Pfaffian closure of an o-minimal structure. 
\textit{J. Reine Angew. Math.} 508 (1999), 189--211. 
MR1676876 (2000j:14093) 

\bibitem{Tarski}
Alfred Tarski,
\textit{A Decision Method for Elementary Algebra and Geometry}, 
RAND Corporation, Santa Monica, 1948.

\bibitem{Athipat2013} 
Athipat Thamrongthanyalak, 
Definable smoothing of continuous functions, \textit{Illinois J. Math.} 57 (2013), nº 3, 801--815. MR3275739 

\bibitem{Tu2010}
Loring W. Tu, 
\textit{An Introduction to Manifolds}. Universitext. Springer, New York (2008). xvi+360 pp. ISBN: 978-0-387-48098-5.
MR2723362 (2011g:58001)

\bibitem{Wilkie2005}
Alex J. Wilkie,
Covering definable open sets by open cells. 
O-minimal Structures, Proceedings of the RAAG Summer School Lisbon 2003, Lecture Notes in Real Algebraic and Analytic Geometry (2005). Cuvillier

\bibitem{Wilkie2016}
Alex J. Wilkie,
Complex continuations of $\mathbb{R}_{\text{an,exp}}$-definable unary functions with a diophantine application, 
\textit{J. Lond. Math. Soc.} (2) 93 (2016), no. 3, 547--566. MR3509953

\end{thebibliography}
\end{document}